\documentclass[a4paper,12pt]{scrartcl}
\usepackage[T1]{fontenc}
\usepackage[utf8]{inputenc}
\usepackage{mathtools}
\usepackage{amsthm,amssymb}
\usepackage[cal=boondoxo]{mathalfa}
\usepackage[pdfpagelabels]{hyperref}
\hypersetup{bookmarks,bookmarksopen,bookmarksdepth=2}
\usepackage[a4paper]{geometry}
\usepackage{enumitem}
\usepackage{csquotes}
\usepackage{stmaryrd}
\usepackage{microtype}
\usepackage{mathabx}
\usepackage{bbm}
\usepackage{cite}

\usepackage{hyperxmp}

\hypersetup{%
    pdftitle={ Constructing General Rough Differential Equations through Flow Approximations },
    pdfauthor={Antoine Lejay},
    pdfdate={2021-12-08},
    pdfkeywords={ Rough paths;  Rough differential equation;  Taylor formula; Aromatic trees;  Branched rough paths },
    pdflang={en},
    pdfmetalang={en},
}

\newtheorem{proposition}{Proposition}
\newtheorem{lemma}{Lemma}
\newtheorem{corollary}{Corollary}
\theoremstyle{definition}

\newtheorem{definition}{Definition}
\newtheorem{notation}{Notation}
\newtheorem{convention}{Convention}
\newtheorem{hypothesis}{Hypothesis}

\theoremstyle{remark}
\newtheorem{remark}{Remark}
\newtheorem{example}{Example}

\DeclarePairedDelimiterXPP{\normsup}[1]{}{\|}{\|}{_{\infty}}{#1}
\DeclarePairedDelimiterXPP{\generalnorm}[2]{}{\|}{\|}{_{#1}}{#2}
\DeclarePairedDelimiter{\opnorm}{\vvvert}{\vvvert}
\DeclarePairedDelimiterXPP{\normlip}[1]{}{\|}{\|}{_\mathrm{Lip}}{#1}

\DeclarePairedDelimiter{\bkt}{\langle}{\rangle}
\DeclarePairedDelimiter{\abs}{|}{|}
\DeclarePairedDelimiter{\lie}{[}{]}
\DeclarePairedDelimiter{\floor}{\lfloor}{\rfloor}
\DeclarePairedDelimiter\Paren()
\DeclarePairedDelimiter\Bra[]
\DeclarePairedDelimiterXPP{\lieG}[1]{}{[}{]}{_{\prodG}}{#1}

\newcommand{\rTT}{\mathbb{T}_+}

\DeclarePairedDelimiterX\Graft[1]\langle\rangle{#1}

\newcommand{\prodB}{\mathbin{\blacktriangleright}}
\newcommand{\prodG}{\mathbin{\smalltriangleright}}
\DeclareMathOperator{\expG}{\exp_{\prodG}}
\DeclareMathOperator{\logG}{\log_{\prodG}}

\DeclarePairedDelimiterXPP{\Fact}[1]{}{(}{)}{!}{#1}

\newcommand{\vd}{\,\mathrm{d}}

\newcommand{\dD}{\mathrm{D}}

\newcommand{\ba}{\mathbf{a}}
\newcommand{\bx}{\mathbf{x}}
\newcommand{\by}{\mathbf{y}}

\newcommand{\btau}{\boldsymbol{\tau}}
\newcommand{\bsigma}{\boldsymbol{\sigma}}
\newcommand{\blambda}{\boldsymbol{\lambda}}

\newcommand{\binBCHDG}{\mathbin{\circledast_{\prodG}}}

\newcommand\given{\nonscript\:\delimsize\vert\nonscript\:\mathopen{}} 
\newcommand\SetSymbol[1][]{\nonscript\:#1\vert\nonscript\:\mathopen{}\allowbreak}
\DeclarePairedDelimiterX\Set[1]\{\}{%
\renewcommand\given{\SetSymbol[\delimsize]}#1}

\DeclareMathOperator{\grandO}{O}

\newcommand{\ind}[1]{\mathbbm{1}_{#1}}

\DeclareMathOperator{\id}{\mathrm{id}}

\newcommand{\NN}{\mathbb{N}}
\newcommand{\RR}{\mathbb{R}}
\newcommand{\TT}{\mathbb{T}}

\newcommand{\uT}{\mathrm{T}}
\newcommand{\uU}{\mathrm{U}}
\newcommand{\uV}{\mathrm{V}}
\newcommand{\ug}{\mathrm{g}}

\newcommand{\cC}{\mathcal{C}}
\newcommand{\cCb}{\mathcal{C}_{\mathrm{b}}}
\newcommand{\cT}{\mathcal{T}}
\newcommand{\cA}{\mathcal{A}}
\newcommand{\cL}{\mathcal{L}}

\newcommand{\ci}{\mathcal{i}}
\newcommand{\cD}{\mathcal{D}}
\newcommand{\cW}{\mathcal{W}}
\newcommand{\cWln}{\mathcal{W}_{\leq n}}
\newcommand{\cI}{\mathcal{I}}
\newcommand{\cF}{\mathcal{F}}

\begin{document}

%%% TITLE 
\title{Constructing
    General\\
    Rough Differential Equations\\
through Flow Approximations}

%%% AUTHOR 
\author{Antoine Lejay\thanks{Université de Lorraine, CNRS, Inria, IECL, F-54000 Nancy, France, \texttt{antoine.lejay@univ-lorraine.fr}}}

\date{December 8, 2021}

\maketitle

%% ABSTRACT
\begin{abstract}
    The non-linear sewing lemma constructs flows 
    of rough differential equations from a broad class of approximations called almost flows. 
    We consider a class of almost flows that could be approximated by solutions
    of ordinary differential equations, in the spirit of the backward error analysis.
    Mixing algebra and analysis, a Taylor formula with remainder and a composition formula
    are central in the expansion analysis.
    With a suitable algebraic structure on the non-smooth vector fields to be integrated, 
    we recover in a single framework several results regarding high-order
    expansions for various kinds of driving paths. 
    We also extend the notion of driving rough path. We introduce as an example
    a new family of branched rough paths, called aromatic rough paths modeled
    after aromatic Butcher series.
\end{abstract}
%%

% AMS 60L20; 34A06

%% KEYWORDS

%% INTRODUCTION
\section{Introduction}

Introduced at the end of the 1990s by T.~Lyons \cite{lyons98a}, the theory of
rough paths defines pathwise solutions to stochastic differential equations
driven by Brownian paths and more generally by a large class irregular paths,
random or deterministic.  The core idea is that the Brownian paths need to be
lifted as paths living in a non-commutative truncated tensor algebra. For
general, irregular path, the order of truncation to consider depends on the
regularity of the path. 

Rough differential equations (RDE) generalize the notion of controlled
differential equations in which the controls, or driving paths, are irregular
in time, \textit{e.g., } $\alpha$-Hölder.  Such paths are called \emph{rough
paths}.  Since the seminal work \cite{lyons98a}, several concurrent
constructions have been given to the existence of solutions of~RDE: by solving
a fixed point \cite{lyons98a,lyons07a,lejay_victoir}, as limit of discrete
approximations by A.~M.~Davie~\cite{davie05a}, by approximating the driving
paths with geodesics, by P.~Friz and N.~Victoir \cite{friz2008,friz}, or by
approximating the flows by I.~Bailleul \cite{bailleul12a,bailleul13b}.

In the recent series of works \cite{brault1,brault2,brault3}, we have studied
the various properties of the non-linear sewing lemma which consider an
abstract family of approximation of flows. This extend the results in
\cite{bailleul12a,bailleul13b} and also encompass the approaches from
A.M.~Davie and P.~Friz \& N.~Victoir \cite{davie05a,friz}.  The main point of
the non-linear sewing lemma is to construct approximations of flows which are
\textquote{rectified} as flows.  

A rough path of order $n$ is a path living in a truncated tensor algebra
$\uT_{\leq n}(\uU):=\bigoplus_{k=0}^n \uU^{\otimes k}$.  A broad family of rough paths may be
constructed from approximating smooth paths through by their iterated integrals.  It shares
the same algebraic properties as the iterated integrals that K.T.~Chen pointed
out in the 1950s \cite{chen57}. Some controls of the regularity are also needed
in accordance with the grading of the tensor space.  In view of performing
integration, several alternative and extension of rough paths have been given.
The first one is that of \emph{controlled rough paths} by
M.~Gubinelli~\cite{gub04}.  Later, M. Gubinelli introduced the notion of
\emph{branched rough paths} \cite{gubinelli10a} which are encoded through
trees. 

The notion of branched rough paths and corresponding RDE has subsequently been 
studied in \cite{hairer_kelly,1810.12179,boedihardjo18a,1712.01965,1604.07352}. 
In \cite{1604.07352}, a solution to RDE driven by branched rough paths have been
given that extends the one of A.M.~Davie: a solution is a path $y:[0,T]\to\uV$ characterized by 
\begin{equation}
    \label{eq:intro:2}
    \abs{y_t-\phi_{t,s}(y_s)}\leq C\abs{t-s}^\theta\text{ for some }\theta>1
\end{equation}
for a proper approximation $\phi_{t,s}$ constructed using the \emph{elementary 
differentials} already used in the context of the studies of B-series  
and Runge-Kutta scheme \cite{hairer_10a}. The article~\cite{2003.12626}
study a Runge-Kutta scheme for RDE, while~\cite{2002.10432,1801.02964}
use algebraic results on Taylor series, B-series and the Faà di Bruno formula 
as a core. A variant to branched rough paths consists in using planar trees~\cite{1804.08515}.
While these articles focus mainly on the properties of the driving paths, the algebraic 
properties of the vector fields play a large role here.

In a generic way, we can see a rough path as a path $\bx:[0,T]\to\cW$ taking 
its values in a sub-group of an algebra $\cW$ (either a tensor algebra or an 
algebra of trees). Equivalently, we can see $\bx$ as a $\cW$-valued family $\Set{\bx_{s,t}}_{s\leq t}$
satisfying the \emph{Chen relation} 
\begin{equation}
    \label{eq:intro:chen}
    \bx_{r,s}\bx_{s,t}=\bx_{r,t}\text{ for any }r\leq s\leq t.
\end{equation}

Let us now consider an algebra homomorphism $F$ from $\cW$ to 
the space $\cC$ of continuous functions from a Banach space $\uV$ into itself with compositions. 
We assume that $F[1]=\id$, the identity map.
Such a $F$ is called a \emph{character}.
The Chen relation~\eqref{eq:intro:chen} translates into 
\begin{equation}
    \label{eq:intro:chen:2}
    F[\bx_{s,t}] = F[\bx_{r,s}]\circ F[\bx_{s,t}]
\end{equation}
This relation implies that $F[\bx_{s,t}](a)$ is a \emph{flow}.

The driving path is not necessarily constructed as satisfying the Chen relation
in~$\cW$, but for example in a quotient space $\widetilde{\cW}:=\cW/\cI$ for an ideal $\cI$.
Truncated tensor algebras are of this type.
In such a case, \eqref{eq:intro:chen:2} is no longer true.  
However, there are situations in which it is \textquote{close} to be true in the sense
that for some $\theta>1$, 
\begin{equation}
    \label{eq:intro:3}
    \abs{F[\bx_{s,t}](a)-F[\bx_{r,s}]\circ F[\bx_{s,t}](a)}\leq C\abs{t-s}^\theta
    \text{ for any }s\leq t,\ a\in\uV.
\end{equation}
With a few additional assumptions on $F[\bx_{r,s}]$, the non-linear sewing 
lemma \cite{bailleul12a,brault1,brault2} shows that there exists a flow
$\Set{\phi_{t,s}}_{s\leq t}$ (\textit{i.e.}, $\phi_{t,s}\circ\phi_{s,t}=\phi_{t,r}$
for any $r\leq s\leq t$) 
close to $\Set{F[\bx_{s,t}]}_{s\leq t}$. Hence, the path $y:[0,T]\to\uV$
defined by $y_t=\phi_{t,s}(a)$ actually satisfies \eqref{eq:intro:2}.
Thus, $y$ is a suitable object for being a solution to a generalized
notion of differential equation.

Given a $\uT_{\leq n}(\uU)$-valued rough path $\bx$ with components $\bx^{(i)}\in\uU^{\otimes i}$, $i=0,\dotsc,n$, 
as well as a function $f:\uV\to L(\uU,\uV)$, a Taylor expansion leads us to naturally consider 
\begin{equation}
    \label{eq:intro:davie}
    F[\bx_{s,t}](a) =\sum_{k=0}^n f^{\{k\}}(a)\bx^{(k)}_{s,t}\text{ for any }s\leq t, 
\end{equation}
where $f^{\{k+1\}}=\dD f^{\{k\}}\cdot f$ and $f^{(0)}=\id$.
Such an approximation was considered first in \cite{davie05a} for $n=2$, and then in \cite{friz2008,2002.10432}
for arbitrary orders $n$. 
In particular, in \cite{friz2008}, it is shown that thanks to an astute use 
of the Davie lemma and sub-Riemannian geodesic approximations, \eqref{eq:intro:3}
holds with $\theta=(n+1)/p$ for a $(1/p)$-Hölder rough path with $n\geq \floor{p}$.
The effect of the approximation on the constant $C$ in \eqref{eq:intro:2}
was studied in \cite{boedihardjo,boedihardjo16b}.
Generalizations to branched rough paths with trees of order $n$ also lead to 
similar controls \cite{boedihardjo18a,1604.07352}.

Another approximation is given by setting $F[\bx_{s,t}](a)=z^{s,t}_1$ where $z^{s,t}$ is
the solution to the ordinary differential equation (ODE)
\begin{equation}
    \label{eq:intro:bailleul}
    z^{s,t}_r=a+\int_0^r F[\blambda_{s,t}](z^{s,t}_u)\vd u,\ r\in[0,1]
    \text{ with }\blambda_{s,t}=\log(\bx_{s,t})
    \text{ in }\widetilde{\cW}.
\end{equation}
The approach is developed in \cite{bailleul12a,bailleul15a,bailleul13b}.

When $\phi_{t,s}$ is seen as an approximation, the kind of control in \eqref{eq:intro:3}
is related to \emph{consistency} of the scheme. Under extra hypothesis, the rate
of the numerical Euler scheme obtained by composing the $\phi_{t,s}$ over small time intervals, 
may be deduced from the knowledge of $\theta$ \cite{friz,brault3,boutaib}.

It is important to point out that naive computations do not lead to a proper
assessment of $\theta$ when $n\geq 3$. They actually lead to $\theta=3/p$ whatever 
the order $n$ of the expansion. To reach a rate that depends on $n$, one really needs 
computations of \emph{algebraic nature}. In particular, for $n\geq 3$, 
we are limited to use \emph{weak geometric rough paths}, which 
are obtained as the limits of smooth rough paths naturally lifted
through their iterated integrals. However, it was shown first in \cite{lejay_victoir}
that for the order of regularity $p\in[2,3)$, and then in \cite{hairer_kelly}
and studied subsequently in \cite{1810.12179,1712.01965} using
branched rough paths that one could reduced non geometric rough paths
to geometric ones through suitable isomorphisms.

Beyond the algebraic aspect, another issue raised to get proper controls on \eqref{eq:intro:3}
is to manage the fact that in an expression like \eqref{eq:intro:davie}
involves functions of different order of regularity when $f$ is not smooth. 
One should really take care of the regularity of the functions when 
expanding then through Taylor approximations to keep only the relevant terms.

To overcome these difficulties, we consider an approach based on a  \textquote{backward error analysis}~\cite{reich}.
We consider solving an ordinary differential equation (ODE) of type 
\begin{equation}
    \label{eq:intro:6}
    y_t[\alpha](a)=a+\int_0^t F[\alpha](y_s[\alpha])\vd s,
\end{equation}
for a suitable $\alpha\in\widetilde{\cW}$ which we related to some $\beta\in\widetilde{\cW}$
through 
\begin{equation}
    \label{eq:intro:5}
    F[\beta](a)=y_1[\alpha](a)+\text{ remainder}.
\end{equation}
Or course, $\beta$ and $\alpha$ are linked by a transform of type exponential/logarithmic.
We actually prove a stronger statement: 
\begin{equation*}
    F[\gamma](y_1[\alpha](a))=F[\beta(\alpha,\gamma)](a)+\text{ remainder}
\end{equation*}
for some $\beta(\alpha,\gamma)\in\widetilde{\cW}$. This is the key to control
over composition of type $F[\alpha]\circ F[\beta]$, and the to rectify $\Set{F[\bx_{s,t}]}_{s\leq t}$ as a
flow using the non-linear sewing lemma. An expression of type \eqref{eq:intro:5} is obtained 
by iterating the use of the Newton formula. This connects~\eqref{eq:intro:bailleul} and~\eqref{eq:intro:davie}.

In our context, we then consider that $F$ from $\widetilde{\cW}$ to the space
of continuous functions not as a character in the space of continuous functions
with composition, but as satisfying a relation of type 
\begin{equation*}
    \dD F[\beta]\cdot F[\alpha]=F[\alpha\beta]
\end{equation*}
for a suitable $\alpha$, where $\dD$ is the Fréchet derivative. We call such a map a Newtonian map.

Alternatively, we may consider $F$ as an algebra morphism from $\widetilde{\cW}$
to the algebra of differential operators. 
To use the Newton formula, $F[\alpha]$ shall be a first-order differential
operator. This imposes some limitations on the kind of $\alpha$ to 
consider, which typically shall belong to a Lie algebra. This explains
why weak rough paths~\cite{lyons02b,friz} shall be considered when dealing with high-order expansions
(however, the correspondence between geometric and non-geometric rough paths
stated above may be used \cite{hairer_kelly,bruned20a}).

In this article,
\begin{itemize}[noitemsep,topsep=-\parskip,partopsep=0pt,leftmargin=2em]
    \item We give a Taylor formula with remainder for approximating 
	flows of ODE with non-smooth vector fields defined as algebra
	homomorphisms from $\cW$ to an algebra of functions or of differential operators.
    \item We also provide a composition formula from which we derive
	inequalities of type~\eqref{eq:intro:3}. The composition formula 
	is a key statement to use the Davie lemma as a substitute
	to the Gronwall lemma and therefore construct to almost flows from flows.
    \item We extend the notion of a driving rough paths as a path
	with values in the truncated algebra $\cW_{\leq n}$ derived from quotienting graded algebra $\cW$ 
	up to a given order $n$. The latter result may easily be generalized by replacing $\cW_{\leq n}$
	to a quotient of $\cW$. This is useful when $\cW$ is the superposition of algebras. 
    \item We show how our results encompass some classical results about
	ODE, mainly the Lie-Trotter formula as well as controls of type \eqref{eq:intro:3}
	for rough differential. 
    \item We obtain decaying (in $n$) bounds similar to the ones of \cite{boedihardjo}
	on the constant in \eqref{eq:intro:3} even for driving path living 
	in an infinite dimensional Banach space as we do not rely on using
	geodesic approximations.
    \item Finally, we introduce a novel notion of \emph{aromatic rough paths}. 
	This class contains the class branched rough paths which itself contains the class of tensor-valued rough paths.
	This is a natural extension to the context
	of rough path of
aromatic Butcher series introduced to construct affine equivariant numerical integrators
	by Munthe-Kaas \textit{et al.} \cite{bogfjellmo,MT2016,MR3510021}.
	\end{itemize}

Here, we focus mostly on constructing suitable vector fields, from which we deduce 
the properties a driving path shall satisfy. The literature mentioned above
focus first on a suitable structure for the driving path.
RDE are defined afterwards.

Beyond aromatic rough paths, other rough paths of similar nature and encoded by
trees could be considered. The Grossmann-Larson (or Cayley) algebra
\cite{grossman} encodes the way differential operators acts on themselves and
leads to a slightly different structure than the one used for the aromatic
B-series. However, one may expect that all the structures lead to similar
results. The prevalence of trees in all these representations is
understood by the fact that trees are convenient ways to encode 
both operations on iterated integrals and on composition of differential operators. 
Aromatic trees exhibit some universal properties \cite{2002.05718}, so that
we could think they lead to the broadest natural class of branched rough paths.
The relationship between these different constructions will be subject to
future work.

\textbf{Outline.} 
In Section~\ref{sec:taylor}, we present the Taylor and the composition formula 
with remainder as well as the suitable algebraic setup related to Newtonian
maps and operators.
In Section~\ref{sec:flow}, we gives a general construction of flows from
rough paths, and gives the suitable controls. 
We show how to apply these results to ODE in Section~\ref{sec:ode}
and to RDE in Section~\ref{sec:rde}. Section~\ref{sec:BRP}
is devoted to branched rough paths, and the  construction of
aromatic branched rough paths.

%%%%%%%%%%%%%%%%%%%%%%%%%%%%%%%%%%%%%%%%%%%%%%%%%%%%%%%%%%%%%%%%%%%%%%
%%%%%%%%%%%%%%%%%%%%%%%%%%%%%%%%%%%%%%%%%%%%%%%%%%%%%%%%%%%%%%%%%%%%%%
\section{The Taylor and the composition formula with remainder}

\label{sec:taylor}

\subsection{The dual nature of ODEs}

Let $\uV$ be a Banach space. The Fréchet derivative operator is denote by $\dD$.
As usual, we denote by $\cC^k(\uV,\uV)$ the space of functions which are continuously differentiable
from $\uV$ to $\uV$ with derivatives of order $k\geq0$.

Let $\cL$ be a set of elements, considered as letters.  
Let $F$ be a map from $\cL$ to the class of Lipschitz functions from $\uV$ to $\uV$.
The ODE 
\begin{equation}
    \label{eq:ode:1}
    y_t[\alpha](a)=a+\int_0^t F[\alpha](y_s[\alpha](a))\vd s\text{ for } t\geq 0
\end{equation}
has a unique $\uV$-valued continuous solution $y[\alpha](a)$ for any $\alpha\in\uV$. 

The Newton formula applied on a function $g\in\cC^1(\uV,\uV)$ yields that 
\begin{equation}
    \label{eq:ode:2}
    g(y_t[\alpha](a))=g(a)+\int_0^t (\dD g(y_s[\alpha](a))\cdot F[\alpha](y_s[\alpha](a))\vd s\text{ for }t\geq 0.
\end{equation}

We define the first order differential operator $F^\dag[\alpha]$ by 
\begin{equation}
    \label{eq:ode:3}
    F^\dag[\alpha]g(a):=\dD g(a)\cdot F[\alpha](a)\text{ for } a\in\uV,\ g\in\cC^1(\uV,\uV),
\end{equation}
as well as the linear differential operator 
\begin{equation}
    \label{eq:ode:9}
    Y_t[\alpha]g(a):=g(y_t[\alpha](a))\text{ for }a\in\uV,\ g\in\cC^0(\uV,\uV). 
\end{equation}

The Newton formula \eqref{eq:ode:2} is then rewritten 
\begin{equation*}
    Y_t[\alpha]g(a)=g(a)+\int_0^t Y_s[\alpha]F^\dag[\alpha]g(a)\vd s\text{ for } a\in\uV,\ g\in\cC^1(\uV,\uV).
\end{equation*}
In the functional form, this means the linear equation
\begin{equation}
    \label{eq:ode:4}
    Y_t[\alpha]=1 +\int_0^t Y_s[\alpha]F^\dag[\alpha]\vd s, 
\end{equation}
where $1$ is the neutral operator $1g=g$ for $g\in\cC^1(\uV,\uV)$.

Conversely, given a differential operator $F^\dag[\alpha]$ with smooth coefficients,   
a Picard approximation applied to \eqref{eq:ode:4} leads to 
\begin{equation}
    \label{eq:ode:5}
    Y_t[\alpha]=1+tF^\dag[\alpha]+\frac{t^2}{2!}F^\dag[\alpha]^2+\frac{t^3}{3!}F^\dag[\alpha]^{3}+\dotsb
    \text{ for }t\geq 0,
\end{equation}
where $F^\dag[\alpha]^k:=F^\dag[\alpha](F^\dag[\alpha]^{k-1})$ for $k\geq 2$.
For convenience, we set $F^\dag[\alpha]^0=1$.

We now assume that $\alpha\mapsto F^\dag[\alpha]$ is a homomorphism from an algebra $\cW\supset \cL$
to the algebra  $\cD$ of differential operators with smooth coefficients. We rewrite the expansion~\eqref{eq:ode:5}
as 
\begin{equation}
    \label{eq:ode:6}
    Y_t[\alpha]=F^\dag[\exp(t\alpha)]\text{ for }t\geq 0\text{ with }\exp(t\alpha)=\sum_{k\geq 0}\frac{t^k}{k!}\alpha^k.
\end{equation}

Let $\ci$ be the identity map from $\uV$ to $\uV$. 
When $F^\dag[\alpha]$ and $F[\alpha]$ are related by  \eqref{eq:ode:3}, 
then $Y_t[\alpha]g=g(y_t[\alpha])$, $t\geq 0$, where $y[\alpha]$ solves \eqref{eq:ode:1}
for any starting point $a\in\uV$. 
As soon as $k\geq 2$, $F^\dag[\alpha]^k$ is in general a differential operator 
of degree~$k$.

Alternatively, when dealing with functions $F[\alpha]$ instead of the differential form $F^\dag[\alpha]$,
the condition $F^\dag[\alpha\beta]=F^\dag[\alpha]F^\dag[\beta]$ is replaced by the following
condition on the Lie derivative:
\begin{equation}
    \label{eq:ode:8}
    F[\alpha\beta](a)=\dD F[\beta](a)\cdot F[\alpha](a)\text{ for }a\in\uV,\ \beta\in\cW,\ \alpha\in\cL\cap\cW.
\end{equation}

Comparing \eqref{eq:ode:6} with \eqref{eq:ode:9} means that $Y_t[\alpha]$ may 
be understood as either a non-linear function acting on functions by composition 
or as a differential operator, hence the dual nature. In the context 
of numerical approximations, this duality is expressed through the notion
of S-series, which generalizes the notion of B-series~\cite{murua}.
Considering functions circumvent the constraint that $F^\dag[\alpha]$ shall be first 
order differential operator.

%%%%%%%%%%%%%%%%%%%%%%%%%%%%%%%%%%%%%%%%%%%%%%%%%%%%%%%%%%%%%%%%%%%%%%
\subsection{Setup for vector fields with non-smooth coefficients}

The expansion in \eqref{eq:ode:6} is only valid if the function $F[\alpha]=F^\dag[\alpha]\ci$
is smooth. This is not the case in many practical situations. We present a setup 
to deal with regular yet not smooth functions.

\begin{hypothesis}
    \label{hyp:100}
    The algebra $\cW$ is a graded, unital algebra $\cW=\bigoplus_{k\geq 0}\cW_k$
    with $\cW_0\simeq \RR$ and $ab\in\cW_{i+j}$ whenever $a\in\cW_i$ and $b\in\cW_j$.
    The neutral element $1$ belongs to $\cW_0$.
\end{hypothesis}
\begin{notation}
    Fix $n\geq 0$. 
    We define $\cW_{\leq n}:=\bigoplus_{k=0}^n \cW_{k}$ and $\cW_{>n}:=\bigoplus_{k\geq n+1} \cW_k$.
    Since $\cW_{>n}$ is a two-sided ideal with respect to the multiplication of $\cW$, 
    $\cW_{\leq n}$ is identified with the quotient space $\cW/\cW_{>n}$. 
\end{notation}
\begin{notation}
    We decompose $\alpha\in\cW$ as $\alpha=\sum_{i=0}^{+\infty} a^{(i)}$ with $\alpha^{(i)}\in\cW_i$, $i\in\NN$.
\end{notation}
\begin{notation}[Relations $\prodG$ and $\prodB$]
    For $\alpha,\beta\in\cW$, we define
    \begin{align}
	\alpha\prodG\beta&:=\sum_{0\leq i+j\leq n}\alpha^{(i)}\beta^{(j)}\in\cW_{\leq n}
	\\
	\text{ and }
	\alpha\prodB\beta&:=\sum_{i+j\geq n+1}\alpha^{(i)}\beta^{(j)}\in \cW_{>n}.
    \end{align}
\end{notation}
\begin{convention}
    \label{conv:1}
    For non-associative algebras, these products are always applied from right to left, for example~$\alpha\prodG\beta\prodG\gamma$
    means $\alpha\prodG(\beta\prodG\gamma)$.
\end{convention}

The product $\prodG$ is the projection of $\alpha\beta$ onto $\cW_{\leq n}$. 
Furthermore, for $\alpha,\beta,\gamma\in\cW$, 
\begin{align}
    \label{eq:313}
    \alpha\prodG\gamma\prodG \beta&=\alpha\prodG (\gamma\beta),\\
    \label{eq:314}
    \alpha\beta&=\alpha\prodG\beta+\alpha\prodB\beta\\
    \label{eq:315}
    \text{and }
    \alpha\prodB\beta&=\sum_{i+j=n+1}^{2n}\alpha^{(i)}\beta^{(j)}\text{ when }\alpha,\beta\in\cW_{\leq n}.
\end{align}

\begin{notation}[Exponential and logarithm]
    We write $\gamma^{\prodG k}$ for $\gamma\in\cW$. 
    The exponential and the logarithm related to $\prodG$ are defined by 
    \begin{align*}
	\expG(\alpha)&:=\sum_{k\geq 0}\frac{1}{k!}\alpha^{\prodG k}\\
	\text{and }
	\logG(1+\alpha)&:=\sum_{k\geq 0}(-1)^{k+1}\alpha^{\prodG k},
    \end{align*}
    for $\alpha\in\cW$ with $\alpha^{(0)}=0$.
\end{notation}
\begin{hypothesis}
    \label{hyp:101}
    There exists a $\RR$-submodule $\cL$ of $\cW$
    which is stable under the projection $\alpha\mapsto \alpha^{(i)}$, for $i=0,\dotsc,n$. 
    We set $\cL_{\leq n}:=\cL\cap\cW_{\leq n}$. 
\end{hypothesis}
%%

%%%%%%%%%%%%%%%%%%%%%%%%%%%%%%%%%%%%%%%%%%%%%%%%%%%%%%%%%%%%%%%%%%%%%%
\subsection{Newtonian map}

Our first definition concerns parametric families of functions. 

\begin{definition}[Newtonian map]
    Assume Hypotheses~\ref{hyp:100} and \ref{hyp:101}.
    A \emph{Newtonian map} is a linear map $F$ from $\cW_{\leq (n+1)}$ to $\cC^0(\uV,\uV)$ such that 
    \begin{itemize}[noitemsep,topsep=-\parskip,partopsep=0pt]
	\item $F[1]=\ci$. 
	\item $F[\alpha]$ is $\cC^1$ for any $\alpha\in\cW_{\leq n}$.
	\item $F[\alpha]$ is Lipschitz continuous for any $\alpha\in\cL_{\leq n}$.
	\item for any $(\alpha,\beta)\in\cL_{\leq n}\times\cWln$, 
	    \begin{equation*}
		\label{eq:ext:2} 
		F[\alpha\beta]:=\dD F[\beta]\cdot F[\alpha]. 
	    \end{equation*}
    \end{itemize}
\end{definition}
\begin{remark} When we combine \eqref{eq:ext:2} with \eqref{eq:314} and \eqref{eq:315}, we obtain
    for $(\alpha,\beta)\in\cL_{\leq n}\times\cWln$, 
    \begin{equation*}
	F[\alpha\beta]=F[\alpha\prodG\beta]+F[\alpha\prodB\beta]        
	=F[\alpha\prodG\beta]+R
    \end{equation*}
    where $\alpha\prodG\beta\in\cW_{\leq n}$ and 
    \begin{equation*}
	R:=\sum_{i+j=n+1}^{2n}\dD F[\beta^{(j)}]\cdot F[\alpha^{(i)}]
    \end{equation*}
    since $\alpha^{(i)}\in\cL_{\leq n}$ by Hypothesis~\ref{hyp:101}. 
    Thus, $F[\alpha\prodG\beta]\in\cC^1$ and $R\in\cC^0$. Thus justifies 
    to use both $\prodG$ and $\prodB$ to keep everything well defined.
\end{remark}

Let us give a few examples. Elementary differentials, initially related to Runge-Kutta numerical schemes \cite{hairer_10a}, 
provides us with another natural family of Newtonian maps which we develop later in Section~\ref{sec:BRP}. 

\begin{example}
    \label{ex:algebra}
    Let $\cF$ be the vector space of functions of class $\cC^\infty(\uV,\uV)$ which we equip with 
    the product $f\ast g:=(a\mapsto \dD g(a)\cdot f(a))$. Thus, $(\cF,+,\ast)$ becomes a non-associative algebra
    with a right unit $\ci:\uV\to\uV$, the identify map on $\uV$.
    Any  algebrea homomorphism $F$ from $(\cW,+,\cdot)$ to $(\cF,+,\ast)$ is a Newtonian map.
\end{example}

Instances of Example~\ref{ex:algebra} are given by analytic functions
and by linear operators.

\begin{example}
    Let us consider as $\cW$ the algebra of $\RR$-valued sequences $\alpha:=\Set{\alpha_i}_{i\in\NN}$
    with the product
    \begin{equation*}
	\alpha\beta:=\Set*{\sum_{i+j=k} j\beta_{j+1}\alpha_i}_{k\in\NN}.
    \end{equation*}
    We set $\cL:=\cW_{>0}$.  For $\uV=\RR$, the function $F[\alpha](a)=\sum_{i\in\NN} \alpha_i a^i$
    is then a Newtonian map, as 
    $\dD F[\beta](a)=\sum_{i\in\NN}i\beta_{i+1}a^i$ and $\dD F[\beta](a)\cdot F[\alpha](a)=F[\alpha\beta](a)$.
\end{example}
\begin{example}
    \label{ex:product}
    We consider $\cW$ is a Banach algebra and we let $\uV:=\cW$.
    For $\alpha\in\cW$, we define $F[\alpha](a):=a \alpha $ for $a\in\uV$.
    We note that $\dD F[\alpha](a)\cdot h= h\alpha $ for any $h\in\uV$. Thus,
    \begin{equation*}
	F[\beta]\ast F[\alpha](a):=\dD F[\beta](a)\cdot F[\alpha](a)=a \alpha\beta=F[\alpha\beta](a).
    \end{equation*}
    Then $F$ is a Newtonian map with $\cL=\cW$.
\end{example}
\begin{remark}
    In general, the unit element $1$ of $\cW$ cannot belong to $\cL$, 
    as $F[1\cdot \beta](a)=\dD F[\beta](a)\cdot a$ and this should be equal to $F[\beta]$
    while $F[1]=\ci$. This is however possible in Example~\ref{ex:product}.
\end{remark}
\begin{example}[Derivatives of Newtonian maps]
    Let $F$ be a Newtonian map. Then $G[\alpha]$ defined by 
    \begin{equation*}
	G[\alpha](a,b)=\begin{bmatrix}
	    F[\alpha](a)\\
	    \dD F[\alpha](a)\cdot b 
	\end{bmatrix}
    \end{equation*}
    is a Newtonian map whenever the regularity is suitable. 
    This statement implies that one may construct 
    flows and their derivatives at the same time.
\end{example}
%%

%%%%%%%%%%%%%%%%%%%%%%%%%%%%%%%%%%%%%%%%%%%%%%%%%%%%%%%%%%%%%%%%%%%%%%

\subsection{Newtonian operators}

\begin{definition}[Differential operator]
    Fix $m\geq 0$.
    A \emph{$m$-order differential operator} $H$ is a linear operator of type 
    \begin{equation*}
	Hg(a)=\sum_{k=0}^m \dD^k g(a)\cdot H_k(a)\text{ for any }g\in\cC^m(\uV,\uV),
    \end{equation*}
    where $H_k$ is a continuous function with values in the symmetrization of $\uV^{\otimes k}$, 
    and $H_m$ is non-vanishing.  We use the convention that $\dD^0 g=g$.
\end{definition}
\begin{definition}[Suitable function]
    Given a differential operator~$H$, we say that~$g$ is a \emph{suitable function} 
    when the regularity of~$g$ is greater or equal than the order of~$H$ so that~$Hg$ is 
    a continuous function.
\end{definition}
\begin{definition}[Newtonian operator]
    Assume Hypotheses~\ref{hyp:100} and \ref{hyp:101}.
    A \emph{Newtonian operator} is a linear map $F^\dag$ from $\cW_{\leq (n+1)}$ to 
    the class of differential operator such that  
    \begin{itemize}[noitemsep,topsep=-\parskip,partopsep=0pt]
	\item $F^\dag[1]=1$, where $1$ is the neutral operator $1g=g$ for any $g\in\cC^0$. 
	\item $F^\dag[\beta]$ has coefficients of class $\cC^1$ for any $\beta\in\cW_{\leq n}$.
	\item $F^\dag[\alpha]g=\dD g\cdot F[\alpha]$ for Lipschitz continuous function $F[\alpha]$ for any $\alpha\in\cL_{\leq n}$.
	\item for any $(\alpha,\beta)\in\cL_{\leq n}\times\cWln$, 
	    \begin{equation*}
		\label{eq:ext:2op} 
		F^\dag[\alpha\beta]g:=\dD F^\dag[\beta]g\cdot F[\alpha]=F^\dag[\alpha]F^\dag[\beta]g,
	    \end{equation*}
	    for any suitable function $g:\uV\to\uV$.
    \end{itemize}
\end{definition}

When $F^\dag[\alpha]$ and $F^\dag[\beta]$ are two first order differential operators,
their product is not a first order differential operator in general. 
The proof of the next lemma is classical \cite[Section III.5.2, p.~85]{hairer_10a}
(it relies on the fact that $\dD^2 g\cdot a\otimes b
=\dD^2 g\cdot b\otimes a$ for any $a,b\in\uV$ and any function $g:\uV\to\uV$ of class $\cC^2$). 

\begin{lemma}
    Set $\lieG{\alpha,\beta}:=\alpha\prodG\beta-\beta\prodG\alpha$. 
    When $\alpha,\beta\in\cL_{\leq n}$, then for any suitable function $g$, 
    \begin{equation*}
	F^\dag[\lieG{\alpha,\beta}]g=\dD g \cdot (F^\dag[\alpha]F[\beta]-F^\dag[\beta]F[\alpha]) 
    \end{equation*}
    so that $F^\dag[\lieG{\alpha,\beta}]$ is a first order differential operator without
    zero order term.
\end{lemma}

The proof of the next three lemmas are immediate.

\begin{definition}[Related Newtonian operators and map]
    \label{def:related}
    A Newtonian map $F$ and a Newtonian operator $F^\dag$ are related when
    $F^\dag$ is the Lie derivative with respect to~$F$, \textit{i.e.},
    \begin{equation}
	\label{eq:316}
	F^\dag[\alpha]g(a)=\dD g(a) \cdot F[\alpha](a)\text{ for any }\alpha\in\cL_{\leq n}, 
	\ a\in\uV \text{ and }g\in\cC^1(\uV,\uV).
    \end{equation}
\end{definition}
\begin{lemma}
    \label{lem:equiv}
    Assume that the Newtonian map $F$ and the Newtonian operator $F^\dag$ are related. 
    Then $g(y_t[\alpha])=Y_t[\alpha]g$, where $y[\alpha]$ solves~\eqref{eq:ode:1}
    and $Y[\alpha]$ solves~\eqref{eq:ode:4}.
\end{lemma}
\begin{lemma}
    \label{lem:OtoM}
    Let $F^\dag$ be a Newtonian operator. Then $F[\alpha]:=F^\dag[\alpha]\ci$ is a Newtonian map
    and \eqref{eq:316} is true.
\end{lemma}
\begin{lemma}
    \label{lem:MtoO}
    Let $F[\alpha]$ be defined as a $\cC^n$ function which is Lipschitz on $\uV$
    for any $\alpha$ in a set $\cL$. 
    Let $\cW$ be the algebra freely generated by the elements of $\cL$.
    We assume that $\cW$ is a graded algebra (the grading is not necessarily 
    the one that arises from the number of terms in the monomials). 
    We define $F^\dag$ by $F^\dag[\alpha]g(a):=\dD g(a)\cdot F[\alpha]$ for $\alpha\in\cL$
    and recursively $F^\dag[\alpha\gamma]g(a):=F^\dag[\alpha]F^\dag[\gamma]g(a)$
    for any $\gamma=\alpha_1\dotsb\alpha_m$ with~$\gamma\in\cW_{\leq n}$. 
    Then~$F^\dag$ defines a Newtonian operator.
\end{lemma}
%%

%%%%%%%%%%%%%%%%%%%%%%%%%%%%%%%%%%%%%%%%%%%%%%%%%%%%%%%%%%%%%%%%%%%%%%
\subsection{The Taylor formula with remainder}

The now state the Taylor formula.

\begin{lemma}[Taylor formula with remainder, functional form]
    Let $F$ be a Newtonian map.
    For any $\alpha\in\cL_{\leq n}$ and $\beta\in\cW_{\leq n}$,
    \begin{equation}
	\label{eq:T1}
	F[\beta](y_t[\alpha](a))=F[\expG(t\alpha)\prodG \beta](a)+R_t[\alpha,\beta](a)
	\text{ for }t\geq 0, 
    \end{equation}
    where $y[\alpha]$ solves \eqref{eq:ode:1} and 
    \begin{equation*}
	R_t[\alpha,\beta](a)=\int_0^t F\Bra*{\alpha\prodB \Paren*{\expG((t-s)\alpha)\prodG \beta}}(y_s[\alpha](a))\vd s
	\text{ for }t\geq 0. 
    \end{equation*}
\end{lemma}
\begin{proof}
    Using the Newton formula applied to  \eqref{eq:ode:1}, since $F[\beta]$ is of class $\cC^1$, 
    \begin{multline*}
	F[\beta](y[\alpha]_t(a))=F[\beta](a)
	+\int_0^t \dD F[ \beta] F[\alpha](y_s[\alpha](a))\vd s
	\\
	=
	F[ \beta](a)+tF[\alpha\prodG \beta](a)
	+\int_0^t F[\alpha\prodB \beta](y_s[\alpha](a))\vd s
	\\
	+\int_0^t \Paren*{F[\alpha \prodG \beta](y_s[\alpha](a))-F[\alpha\prodG \beta](a)}\vd s.
    \end{multline*}
    Iterating this procedure, 
    \begin{multline*}
	F[ \beta](y[\alpha]_t(a))
	\\
	=
	\sum_{k=0}^{+\infty} \frac{t^k}{k!}F[{\alpha^{\prodG k}\prodG \beta}](a)
	+\sum_{k=1}^{+\infty}
	\int_0^t \int_0^{t_1}\dotsb \int_0^{t_{k-1}}
	F[\alpha\prodB\alpha^{\prodG (k-1)}\prodG \beta](y_{t_k}[\alpha](a))\vd s.
    \end{multline*}
    Using for example \cite[Exercise~2.81, p.~253]{duistermaat}, 
    \begin{multline*}
	\int_0^t \int_0^{t_1}\dotsb \int_0^{t_{k-1}}
	F[\alpha \prodB\alpha^{\prodG (k-1)} \prodG \beta](y_{t_k}[\alpha](a))\vd t_k\dotsb\vd t_1
	\\
	=
	\int_0^t \frac{(t-s)^{k-1}}{(k-1)!}F[\alpha\prodB\alpha^{\prodG (k-1)} \beta](y[\alpha]_{s}(a))\vd s.
    \end{multline*}
    With the definition of the exponential $\expG$, this gives the desired formula.
\end{proof}

The proof of the operational form is similar so that we skip it.

\begin{lemma}[Taylor formula with remainder, operational form]
    Let $F^\dag$ be a Newtonian operator.
    For any $\alpha\in\cL_{\leq n}$ and $\beta\in\cW_{\leq n}$,
    \begin{equation}
	\label{eq:T2}
	Y_t[\alpha]F^\dag[\beta]=F^\dag[\expG(t\alpha)\prodG \beta]+R_t^\dag[\alpha,\beta],
    \end{equation}
    where $Y[\alpha]$ is defined by \eqref{eq:ode:9} and 
    \begin{equation*}
	R_t^\dag[\alpha,\beta]g(a):=\int_0^t Y_s[\alpha]F^\dag\Bra*{\alpha\prodB \Paren*{\expG((t-s)\alpha)\prodG \beta}}g(a)\vd s,
    \end{equation*}
    for any suitable function $g$. 
\end{lemma}

\begin{remark}
    For $b,b'\in\uV$, we consider $\cL=\Set{\alpha}$ for a symbol $\alpha$
    and $\cW$ the word algebra generated, with $\cW_k$ the span of $\alpha^k$.
    We set $F^\dag[\alpha]g(a):=\dD g(a)\cdot (b'-b)$. 
    We define for $g\in\cC^k$ and $a\in\uV$,
    \begin{equation*}
	F^\dag[\alpha^k]g(a):=F^\dag[\alpha]F^\dag[\alpha^{k-1}]g(a)=\dD^k g(a)\cdot (b'-b)^{\otimes k}
	\text{ for }a\in\uV.
    \end{equation*}
    The Taylor formula \eqref{eq:T2} contains the usual Taylor formula with remainder
    for any~$g\in\cC^{n+1}$:
    \begin{equation*}
	g(b')=\sum_{k=0}^n \frac{1}{k!} \dD^k g(b)\cdot (b'-b)
	+\int_0^1 \frac{(1-s)^{n}}{n!}\dD^{n+1} g(sb+(1-s)b')\cdot (b'-b)^{\otimes k}\vd s
    \end{equation*}
    since  $\alpha\prodB\expG((t-s)\alpha)=(t-s)^n\alpha^{n+1}/n!$.
\end{remark}

%%%%%%%%%%%%%%%%%%%%%%%%%%%%%%%%%%%%%%%%%%%%%%%%%%%%%%%%%%%%%%%%%%%%%%
\subsection{The composition formula}

The subspace $\cL$ is not necessarily a subalgebra.
We start by defining a binary relation between elements of $\cL$. 

\begin{notation}
    \label{not:301}
    We write 
    \begin{align}
	\label{eq:309}
	\alpha\binBCHDG\beta&:=\logG(\expG(\alpha)\prodG\expG(\beta))\text{ for }\alpha,\beta\in\cL.
    \end{align}
    By the properties of $\prodG$, $\alpha\binBCHDG\beta\in\cL_{\leq n}$.
\end{notation}

Usually, expressions of $\alpha\binBCHDG\beta$ is given by the truncated Baker-Campbell-Hausdorff-Dynkin (BCHD) formula
\cite{bonfiglioli}. BCHD type formula can be extended to  non-associative algebras \cite{mostovoy17a}.

\begin{hypothesis}
    \label{hyp:303}
    The submodule $\cL$ is such that $\alpha\binBCHDG\beta\in\cL_{\leq n}$ for any $\alpha,\beta\in\cL$.
\end{hypothesis}

\begin{example}
    When $\cW$ is a graded Lie algebra and $\cL$ is a Lie algebra generated by $\lie{\alpha,\beta}=\alpha\beta-\beta\alpha$ 
    containing $\cW_1$, then Hypothesis~\ref{hyp:303} holds true.
\end{example}

\begin{lemma}[Composition formula, functional form]
    Let $F$ be a Newtonian map.
    For $\alpha,\beta\in\cL_{\leq n}$, under Hypotheses~\ref{hyp:100}, \ref{hyp:101} and~\ref{hyp:303}, 
    \begin{multline}
	\label{eq:302}
	y_1[\beta]\circ y_1[\alpha](a)
	=y_1[\alpha\binBCHDG \beta](a)
	\\
	+R_1[\alpha,\expG(\beta)](a)
	+R_1[\beta,1](\Phi[\alpha](a))
	-R_1[\alpha\binBCHDG \beta,1](a).
    \end{multline}
\end{lemma}
\begin{proof}
    Using the Taylor formula on $y_1[\beta]$ with $F[1]$, 
    \begin{equation*}
	y_1[\beta]\circ y_1[\alpha](a)=F[\expG(\beta)](y_1[\alpha](a))+R_1[\beta,1](y_1[\alpha](a)).
    \end{equation*}
    Applying again the Taylor formula on $y_1[\alpha]$ with $F[\expG(\beta)]$, 
    \begin{equation*}
	y_1[\beta]\circ y_1[\alpha](a)
	=
	F[\expG(\alpha)\prodG \expG(\beta)](a)
	+R_1[\alpha,\expG(\beta)](a)
	+R_1[\beta,1](y_1[\alpha](a)).
    \end{equation*}
    Finally, using the Taylor formula on $y_1[\alpha\binBCHDG\beta]$ with $f[1]$ leads to~\eqref{eq:302}.
\end{proof}

We now state the operational counterpart of the composition formula.

\begin{lemma}[Composition formula, operational form]
    Let $F^\dag$ be a Newtonian operator.
    For $\alpha,\beta\in\cL_{\leq n}$, under Hypotheses~\ref{hyp:100}, \ref{hyp:101} and~\ref{hyp:303}, 
    \begin{multline}
	\label{eq:302bis}
	Y_1[\alpha]Y_1[\beta]g
	=Y_1[\alpha\binBCHDG \beta]g
	\\
	+R_1^\dag[\alpha,\expG(\beta)]g
	+R_1^\dag[\beta,1]g\circ Y_1[\alpha]
	-R_1^\dag[\alpha\binBCHDG\beta,1]g
    \end{multline}
    for any suitable function $g$.
\end{lemma}
%%

%%%%%%%%%%%%%%%%%%%%%%%%%%%%%%%%%%%%%%%%%%%%%%%%%%%%%%%%%%%%%%%%%%%%%%
\section{Flows of differential equations controlled by rough paths}

\label{sec:flow}

\subsection{Building $\cW$-valued rough paths}

As introduced by T.~Lyons in \cite{lyons98a}, a $p$-rough path is a path with values
in the truncated tensor algebra $\uT_{\leq n}(\uU)$ of a Banach space $\uU$
and of finite $p$-variation. Later, M.~Gubinelli constructed a rough 
path with values in the algebra of rooted trees~\cite{gubinelli10a}.
These constructions extend the notion of differential 
equations.

The notion of rough paths itself is easily extended to a general graded Banach algebra.

\begin{definition}[Control]
    A \emph{control} is a map $\omega:\rTT^2\to\RR_+$ which 
    is super-additive, that is $\omega_{r,s}+\omega_{s,t}\leq \omega_{r,t}$
    and continuous close to its diagonal.
\end{definition}
\begin{hypothesis}
    The graded algebra $\cW$ is a Banach algebra with a metric $\abs{\cdot}$ which satisfies
    $\abs{\alpha\beta}\leq\abs{\alpha}\cdot\abs{\beta}$ for any $\alpha,\beta\in\cW$
    and 
    \begin{equation*}
	\abs{\alpha}=\sum_{i\geq 0}\abs{\alpha^{(i)}}\text{ where }\alpha^{(i)}\in\cW_i,\ i\geq 0
	\text{ and }\alpha=\sum_{i\geq 0 }\alpha^{(i)}.
    \end{equation*}
\end{hypothesis}
\begin{definition}[$\cW$-valued rough paths]
    A $\cW$-valued family $\bx:=(\bx_{s,t})_{(s,t)\in\rTT^2}$ is a \emph{$\cW$-valued $p$-rough path}
    if there exists a $\RR_+$-valued family $\Set{\mu_i}_{i\in\NN}$ such 
    that 
    \begin{itemize}[noitemsep,topsep=-\parskip,partopsep=0pt]
	\item For $i\geq 0$, $\bx^{(i)}_{s,t}\in\cW_i$ with $\bx^{(0)}_{s,t}=1$ and
	    \begin{equation}
		\label{eq:Wrp:1}
		\abs{\bx^{(i)}_{s,t}}\leq \mu_i\omega_{s,t}^{i/p}\text{ for any }i\geq 1,\ (s,t)\in\rTT^2.
	    \end{equation}

	\item The family $\bx$ is \emph{multiplicative}, that is 
	    \begin{equation}
		\label{eq:Wrp:2}
		\bx_{r,s}\bx_{s,t}=\bx_{r,t}\text{ for any }(r,s,t)\in\rTT^3.
	    \end{equation}

	\item For any $(s,t)\in\rTT^2$, $\abs{\bx_{s,t}}$ is finite.
    \end{itemize}
\end{definition}

Since $\cW_{>n}$ is a two-sided ideal, $\cWln$ is 
isomorphic to the quotient space $\cW/\cW_{>n}$.
The Lyons extension theorem is also easily generalized,

\begin{proposition}[Lyons extension theorem]
    Let $\bx$ be a $\cWln$-valued rough path with $n>\floor{p}$. 
    Then there exists a unique $\cW$-valued rough path $\by_{s,t}$
    such that $\bx_{s,t}=\by^{(\leq n)}_{s,t}$, where
    $\alpha^{(\leq n)}:=\sum_{i=0}^n \alpha^{(i)}$. 
\end{proposition}
\begin{proof}
    The $\cWln$-rough path $\bx$ is almost multiplicative in~$\cW$, that is 
    \begin{equation*}
	\bx_{r,s,t}:=
	\bx_{r,s}\bx_{s,t}-\bx_{r,t}=\bx_{r,s}\prodB\bx_{s,t}\in\cW_{> n}
	\text{ for any }(r,s,t)\in\rTT^3.
    \end{equation*}
    Therefore,
    \begin{equation*}
	\abs{\bx_{r,s,t}}\leq \omega_{r,t}^{(n+1)/p}(n-1)\abs{\bx_{r,t}}^2.
    \end{equation*}
    The multiplicative sewing lemma \cite{lyons98a,feyel} is sufficient to conclude.
\end{proof}
%%

%%%%%%%%%%%%%%%%%%%%%%%%%%%%%%%%%%%%%%%%%%%%%%%%%%%%%%%%%%%%%%%%%%%%%%
\subsection{Rough Differential Equations driven by $\cW$-valued rough paths}

\begin{hypothesis}
    \label{hyp:304}
    Fix $n>0$.  There exist 
    \begin{itemize}[noitemsep,topsep=-\parskip,partopsep=0pt]
	\item A submodule $\cL$ of $\cW$ satisfying Hypothesis~\eqref{hyp:101}.
	\item For some $p\geq 1$ and $n\geq \floor{p}$, a $\cW$-valued $p$-rough path $\bx$ 	
	    with $\blambda_{s,t}:=\logG(\bx_{s,t})\in\cL_{\leq n}$.
	\item A Newtonian map $F$.
    \end{itemize}
\end{hypothesis}
\begin{hypothesis}
    \label{hyp:305}
    There exists  $\Set{\nu_i}_{i=1,\dotsc,n}\in\RR^n_+$ such that for any $\alpha\in\cW_n$, 
    \begin{equation*}
	\max\Set*{\normlip{F[\alpha^{(i)}]},\normsup{F[\alpha^{(i)}]}}\leq \nu_i \abs{\alpha^{(i)}}\text{ for }i=1,\dotsc,n.
    \end{equation*}
\end{hypothesis}
\begin{notation}
    We define a family $\phi$ of maps from $\uV$ to $\uV$ by 
    \begin{equation*}
	\phi_{t,s}(a):=y[\blambda_{s,t}]_1(a) \text{ for } a\in\uV,\ (s,t)\in\rTT^2.
    \end{equation*}
\end{notation}
\begin{proposition}
    \label{prop:GDE}
    Under Hypotheses~\ref{hyp:304} and \ref{hyp:305}, 
    if $\cW$ is finite dimensional, then for any $a\in\uV$, 
    there exists a path $y\in\cC^0(\TT,\uV)$, not necessarily unique, 
    such that $y_0(a)=a$ and 
    \begin{equation}
	\label{eq:312}
	\abs{y_t-\phi_{t,s}(y_s)}\leq C\omega_{s,t}^{(n+1)/p}
	\text{ for }(s,t)\in\rTT^2.
    \end{equation}

    Such a path is called a \emph{D-solution} to 
    \begin{equation}
	\label{eq:311}
	y_t=a+\int_0^t F[\vd \bx_{s}](y_s)\vd s\text{ for }t\geq 0.
    \end{equation}
\end{proposition}

\begin{remark}
    Although it is possible to give an explicit bound for $C$
    from the following proof, this constant is improved
    in Section~\ref{sec:decay}.
\end{remark}

\begin{notation}
    \label{not:generalnorm}
    We denote by $\cCb^n$ the set of functions of class $\cC^n$ with are bounded
    with bounded, continuous derivatives up to order $n$. For $g\in\cCb^n$, 
    we set $\generalnorm{\leq n}{g}:=\sum_{k=1}^n \normsup{\dD^k g}$.
\end{notation}

\begin{corollary}[Newton's formula]
    \label{cor:GDE}
    Assume Hypotheses~\ref{hyp:304} and \ref{hyp:305}. Let $F^\dag$ be the Newtonian 
    map related to $F$ through~\eqref{eq:316}. Then for any function $g$ of class $\cCb^n$, 
    \begin{equation*}
	\abs*{g(y_t)-F^\dag\Bra*{\bx^{(\leq n)}_{s,t}}g(y_s) }\leq C\generalnorm{\leq n}{g}\omega_{s,t}^{(n+1)/p}
	\text{ for }(s,t)\in\rTT^2 
    \end{equation*}
    for a constant $C$ 
    whenever $\normsup{F^\dag[\alpha]g}\leq K\abs{\alpha}\cdot\generalnorm{\leq n}{g}$
    for some constant~$K\geq 0$ and any~$\alpha\in\cL_{\leq n}$.
\end{corollary}

We recall the definition of an \emph{almost flow} (or \emph{approximate flow} in \cite{bailleul12a}), 
here in the context of a Banach space. The definition of \cite{brault1} is slightly weaker. 
For the sake of simplicity, 
we restrict ourselves to bounded functions/vector fields: 
see \textit{e.g.} \cite{brault3} for controls allowing to remove such an assumption.
Also, we may weaken the control over the Lipschitz norm of $F[\alpha]$ to a control over 
its Hölder norm in Hypothesis~\ref{hyp:305}.

\begin{definition}[Almost flow]
    A family  $\phi:=\Set{\phi_{s,t}}_{(s,t)\in\rTT^2}$ of maps from $\uV$ to $\uV$
    is an \emph{almost flow} if for some non-decreasing, continuous function $\delta:\RR_+\to\RR_+$
    with $\delta(0)=0$,
    \begin{itemize}[noitemsep,topsep=-\parskip,partopsep=0pt]
	\item $\phi_{t,t}=\ci$ for any $t\in\uV$.
	\item $\normsup{\phi_{t,s}-\ci}\leq \delta_T$ for any $(s,t)\in\rTT^2$.
	\item $\normlip{\phi_{t,s}}\leq 1+\delta_T$ for any $(s,t)\in\rTT^2$.
	\item For some $\theta>1$ and some constant $L\geq0$,  
	    \begin{equation}
		\label{eq:af}
		\normsup{\phi_{t,s}\circ\phi_{s,r}-\phi_{t,r}}\leq L\omega_{r,t}^\theta
		\text{ for any }(r,s,t)\in\rTT^3.
	    \end{equation}
    \end{itemize}
    Besides,  $\phi$ is a \emph{flow} if $L=0$ in \eqref{eq:af}.
\end{definition}

The proof of the next results is proved using an homogeneity argument.
We skip its proof.

\begin{lemma}
    \label{lem:12}
    There exist  some constants $\lambda_i,L_i\geq0$, $i=1,\dotsc,n$ such that 
    \begin{equation}
	\label{eq:310}
	\abs{\blambda_{s,t}^{(i)}}\leq \lambda_i \omega_{s,t}^{i/p}
	\text{ and }
	\abs{\exp(\theta\blambda_{s,t})^{(i)}}\leq L_i\omega_{s,t}^{i/p}
	\text{ for any }(s,t)\in\rTT^2\text{ and }0\leq \theta\leq 1
    \end{equation}
    for $i=1,\dotsc,n$.
\end{lemma}
\begin{proof}[Proof of Proposition~\ref{prop:GDE}]
    We prove first that $\phi$ is an almost flow.
    Since for any $s\in\TT$, $\blambda_{s,s}=0$, $\phi_{s,s}(a)=a$.

    From applications of the Gronwall lemma,
    \begin{align*}
	\normsup{\phi_{t,s}-\ci}\leq N_{s,t}
	\normlip{\phi_{t,s}}\leq N_{s,t}
    \end{align*}
    with 
    \begin{equation*}
	N_{s,t}:=\sum_{i=1}^n \nu_i\abs{\blambda_{s,t}^{(i)}}
	\leq \sum_{i=1}^n \nu_i\lambda_i\omega_{s,t}^{i/p},
    \end{equation*}
    for any $(s,t)\in\rTT^2$. We then set $\delta_T:=N_{0,T}$.

    Since, $\bx_{r,s}\prodG\bx_{s,t}=\bx_{r,t}$ in $\cW_{\leq n}$, 
    it follows from \eqref{eq:309} that $\lambda_{r,s}\binBCHDG \lambda_{s,t}=\lambda_{r,t}$
    for any $(r,s,t)\in\rTT^2$.
    The composition formula yields
    \begin{equation*}
	\phi_{t,s}\circ\phi_{s,r}(a)
	-\phi_{t,r}(a)
	=R[\blambda_{r,s},\bx^{(\leq n)}_{s,t}](a)
	+R[\blambda_{s,t},1](\phi_{r,s}(a))
	-R[\blambda_{r,t},1](a).
    \end{equation*}
    With \eqref{eq:310}, \eqref{eq:Wrp:1} and Hypothesis \eqref{hyp:305}, 
    for some constant $C'$, 
    \begin{equation*}
	\normsup{R[\blambda_{r,s},\bx^{(\leq n)}_{s,t}]}
	\leq \sum_{k=n+1}^{2n}\nu_k\sum_{i+j=k}\Paren*{\sum_{\ell+\ell'=i}\lambda_\ell\Lambda_{\ell'}}
	\mu_i\omega_{r,t}^{i/p}\omega_{s,t}^{j/p}
	\leq C'\omega_{r,t}^{(n+1)/p}
    \end{equation*}
    for any $(r,s,t)\in\rTT^3$.
    With $s=t$, this inequality applies to $\normsup{R[\blambda_{s,t},1]}$.
    Hence, 
    \begin{equation*}
	\normsup{\phi_{t,s}\circ\phi_{s,r}-\phi_{t,r}}\leq C\omega_{r,t}^{(n+1)/p}.
    \end{equation*}
    This proves that $\phi$ is an almost flow. 
    The results follows from \cite[Theorem~2]{brault1}.
\end{proof}
\begin{proof}[Proof of Corollary~\ref{cor:GDE}]
    For any $(s,t)\in\rTT^2$, 
    \begin{multline*}
	\abs*{g(y_t)-F^\dag\Bra*{\bx^{(\leq n)}_{s,t}}g(y_s)}
	\leq 
	\abs{g(y_t)-\phi_{t,s}g(y_s)}
	+\abs*{\phi_{t,s}g(y_s)-F^\dag\Bra*{\bx_{s,t}}g(y_s)}
	\\
	\leq
	\normlip{g}\cdot\abs{y_t-\phi_{t,s}(y_s)}
	+\normsup{R^\dag[\blambda_{s,t},1]g(y_s)}
	\\
	\leq 
	C\normlip{g}\omega_{s,t}^{(n+1)/p}
	+C'\generalnorm{\leq n}{g}\omega_{s,t}^{(n+1)/p}.
    \end{multline*}
    The last inequality follows from Lemma~\ref{lem:12}.
\end{proof}
\begin{remark}
    In the proof of Proposition~\ref{prop:GDE}, when $F$ and $F^\dag$ are related, we have proved that 
    \begin{equation*}
	F^\dag[\bx_{r,s}\bx_{s,t}]\ci(a)=y_1[\blambda_{s,t}](y_1[\blambda_{r,s}](a))+\text{remainder}
    \end{equation*}
    as well as 
    \begin{equation*}
	F^\dag[\bx_{r,s}\bx_{s,t}]\ci(a)=F[\bx_{s,t}](F[\bx_{r,s}](a))+\text{remainder}.
    \end{equation*}
    As the coefficients of $F^\dag[\bx_{s,t}]$ may only be of class $\cCb^1$, 
    $F^\dag[\bx_{r,s}]F^\dag[\bx_{s,t}]$ is not necessarily well defined, while 
    $F[\bx_{s,t}]\circ F[\bx_{r,s}]$ is.
\end{remark}

We now give a sufficient condition to ensure uniqueness. Actually, 
much more results can be given, such as generic properties, 
rate of convergence of numerical schemes and so on as 
for ODE with Lipschitz vector fields \cite{brault1,brault2,brault3}.

\begin{definition}[4-points control]
    A function $g:\uV\to\uV$ satisfies a \emph{4-points control} whenever 
    \begin{equation*}
	\abs{g(a)-g(b)-g(c)+g(d)}
	\leq \widehat{g}(\abs{a-b}\vee\abs{c-d})\times (\abs{a-c}\vee \abs{b-d})
	+g^\oast\abs{a-b-c+d}
    \end{equation*}
    for a non-decreasing, continuous function $\widehat{g}:\RR_+\to\RR_+$ and $C\geq 1$. 
\end{definition}

The proof of the next lemma is a direct consequence of the Gronwall lemma.

\begin{lemma}
    \label{lem:11}
    Let $g$ be a function that satisfies a 4-points control.
    Let $y$ be the solution to the  ODE $y_\tau(a)=a+\int_0^\tau g(y_\sigma(a))\vd \sigma$ 
    (which is necessarily unique since $g$ is Lipschitz continuous). 
    Then $h:=a\mapsto y_1(a)$ satisfies a 4-points control 
    with $\widehat{h}:=\exp(g^\oast)\widehat{g}(\normlip{h})\normlip{h}$ and 
    $h^\oast:=\exp(g^\oast)$, where $\normlip{h}\leq \exp(\normlip{g})$.
\end{lemma}

\begin{proposition}
    \label{prop:301}
    Assume that $F[\blambda_{s,t}]$ satisfies a 4-points control 
    for any $(s,t)\in\rTT^2$ with $F[\blambda_{s,t}]^\oast\leq \delta_T$
    and for some $\theta>1$, 
    \begin{equation}
	\label{eq:105}
	\normlip{R[\blambda_{s,r},\bx_{s,t}^{(\leq n)}]}
	+\normlip{R[\bx_{s,t}^{(\leq n)},1]}\leq C\omega_{r,t}^\theta
	\text{ for any }(r,s,t)\in\rTT^3. 
    \end{equation}
    Then $\Set{\phi_{s,t}}$ is a stable almost flow and the D-solution to \eqref{eq:311} is unique.
    The D-solution also exists in an infinite Banach space.
\end{proposition}

Alternatively, we may consider working on $F[\bx^{(\leq n)}_{s,t}]$ using 
perturbation results \cite{brault1,brault2}.

\begin{proposition}
    The family $\psi$
    defined by  $\psi_{t,s}(a):=F[\bx^{(\leq n)}_{s,t}]$, $(s,t)\in\rTT^2$ 
    is an almost flow with $\normsup{\psi_{s,t}-\phi_{s,t}}\leq C\omega_{s,t}^{\theta\wedge(n+1)/p}$
    for any $(s,t)\in\rTT^2$. If in addition~\eqref{eq:105} holds and~$\phi$ is a stable almost 
    flow, then $\psi$ is also a stable almost flow.
\end{proposition}
%%

%%%%%%%%%%%%%%%%%%%%%%%%%%%%%%%%%%%%%%%%%%%%%%%%%%%%%%%%%%%%%%%%%%%%%%
\subsection{Decay estimate}

\label{sec:decay}

One could see the $\phi_{t,s}$ as numerical integrators and to use 
the almost flow to construct approximation schemes.
This is why \eqref{eq:312}, which corresponds to consistency, 
actually determines the rate of convergence from the knowledge.

An inequality of type~\eqref{eq:312} only state  consistency.  When
$\normlip{\phi_{t,s}}$ is of order $t-s$, for example for ODE, then the
Gronwall lemma gives the rate of convergence.  This is no longer true when
$\phi_{s,t}$ is not of order $t-s$.  In Proposition~\ref{prop:301}, we give a
sufficient condition to get uniqueness and rate of convergence.  The Davie
lemma (See Appendix~\ref{sec:davie}) is a substitute to the Gronwall lemma.
Its requires to already have an
estimate on $\Gamma_{s,t}:=y_t-\phi_{t,s}(y_s)$ of type 
$\abs{\Gamma_{s,t}}\leq C\omega_{s,t}^\theta$ for some $\theta>1$. Given that
$\abs{\Gamma_{r,s,t}}\leq M\omega_{r,t}^\theta$ with
$\Gamma_{r,s,t}=\Gamma_{r,s}+\Gamma_{s,t}-\Gamma_{r,t}$, then $C$ may be
replaced by $KM$ for a constant $K$ depending only on $\omega_{0,T}$ and
$\theta>1$.  This was the strategy proposed first by P.~Friz and N.~Victoir
\cite{friz2008,friz} to study high-order Euler schemes. They however use a geodesic
approximation to get first $\abs{\Gamma_{s,t}}\leq C\omega_{s,t}^\theta$ with a
constant $C$ that depends on the length of an approximation by a smooth path. They
hence obtain a constant $M$  in $\abs{\Gamma_{r,s,t}}\leq M\omega_{r,t}^\theta$
that does not depend on the regularity of the path. 

It is possible to get the constant $C$ explicitly. For this, it relies on
getting bounds on the $\lambda_i$ in \eqref{eq:310}.  However, this is useless
as the constant may be strongly improved.  For this, we follow the general idea
from~\cite{boedihardjo} (see also \cite{boedihardjo18a,boedihardjo18a} for
related works).  The restriction to consider only rough paths with values in
a finite-dimensional space is removed as we do not rely on using a sequence of
smooth approximations of the path.

Here, we consider a Newtonian operator $F^\dag$, as we saw in Lemma~\ref{lem:MtoO} how to transform 
a Newtonian map into a Newtonian operator.  
From now, we denote by~$y$ a D-solution to $y_t=a+\int_0^t F[\vd \bx_r](y_r)\vd r$.

\begin{notation}
    We define \begin{equation}
	\label{eq:317}
	R^{s,t}_j(g):=g(y_t)-F^\dag\Bra*{\bx_{s,t}^{(\leq j)}}g(y_s)
    \end{equation}
    for $j=0,1,2,\dotsc,n$,  $g\in\cC^j$  and any $(s,t)\in\rTT^2$.
\end{notation}
\begin{hypothesis}
    \label{hyp:10}
    Hypothesis~\ref{hyp:305} holds. Besides, 
    there exists a constant $\gamma\geq 0$ such that $n-1+\gamma>p$ and for each $j\in\Set{1,\dotsc,n-1}$ 
    a constant $k_j\geq 0 $ such that 
    \begin{equation}
	\label{eq:73}
	\normsup{R^{s,t}_j(F[\alpha]g)}\leq k_j\abs{\alpha}\cdot\generalnorm{\leq j}{g}\omega_{s,t}^{(\gamma+j)/p}
    \end{equation}
    for any $g\in\cCb^j$ and any $\alpha\in\cW_{\leq j}$.
\end{hypothesis}
\begin{proposition}[Propagation of decay]
    \label{prop:decay}
    Consider Hypothesis~\ref{hyp:10}. We assume that for a constant $A$,
    \begin{equation}
	\label{eq:72}
	\normsup{R^{s,t}_n(g)}\leq A\generalnorm{\leq n}{g}\omega_{s,t}^{(n+\gamma)/p},\text{ for any }0\leq s\leq t\leq T,
    \end{equation}
    for  $g\in\cCb^n$.
    Then \eqref{eq:72} holds true with $A$ replaced by 
    \begin{equation}
	\label{eq:74}
	k_n:=\sup_{\substack{(r,s,t)\in\rTT^3\\r<s<t}}
	\abs{\bx}\dfrac{
	    (1-2^{(p-n-\gamma)/p})^{-1}
	    \sum_{j=1}^n k_{n-j}\nu_j\omega_{r,s}^{(\gamma+n-j)/p}\omega_{s,t}^{j/p}
	}{\omega_{r,t}^{(n+\gamma)/p}}.
    \end{equation}
\end{proposition}
\begin{proof}
    Set for $(r,s,t)\in\rTT^3$ and a suitable function $g:\uV\to\uV$, 
    \begin{equation}
	\label{eq:64}
	V_{r,s,t}:=F^\dag[\bx^{(\leq n)}_{s,t}-1]g(y_s)-F^\dag[\bx^{(\leq n)}_{s,t}-1]g(y_r).
    \end{equation}
    Recall that $F^\dag[1]g=g$ so that $F^\dag[\bx^{(\leq n)}_{s,t}]g(a)=F^\dag[\bx^{(\leq n)}_{s,t}-1]g(a)+g(a)$.

    On the one hand, 
    \begin{equation}
	\label{eq:65}
	R^{r,s}_n(g)+R^{s,t}_n(g)
	=g(y_t)-F^\dag[\bx_{r,s}^{(\leq n)}]g(y_r)-F^\dag[\bx^{(\leq n)}_{s,t}-1]g(y_r)-V_{r,s,t}. 
    \end{equation}
    On the other hand, 
    \begin{multline}
	\label{eq:66}
	R^{r,t}_n(g)
	=g(y_t)-\sum_{\substack{0\leq i+j\leq n\\i,j\geq 0}}F^\dag[\bx_{r,s}^{(i)}\bx_{s,t}^{(j)}]g(y_r)
	\\
	=g(y_t)-\sum_{j=1}^n\sum_{i=1}^{n-j}F^\dag[\bx_{r,s}^{(i)}\bx_{s,t}^{(j)}]g(y_r)
	-F^\dag[\bx^{(\leq n)}_{s,t}]g(y_r)-F^\dag[\bx^{(\leq n)}_{r,s}-1]g(y_r).
    \end{multline}
    With \eqref{eq:317}, 
    \begin{multline*}
	F^\dag[\bx^{(j)}_{s,t}]g(y_s)-F^\dag[\bx^{(j)}_{s,t}]g(y_r)
	=F^\dag[\bx^{(\leq n-j)}_{r,s}]F^\dag[\bx^{(j)}_{s,t}]g(y_r)+R_{n-j}^{r,s}(F^\dag[\bx^{(j)}_{s,t}]g)
	\\
	=\sum_{i=1}^{n-1} F^\dag[\bx^{(j)}_{r,s}\bx^{(i)}_{s,t}]g(y_r)+R_{n-j}^{r,s}(F^\dag[\bx^{(j)}_{s,t}]g). 
    \end{multline*}
    With $V_{s,t}$ introduced in \eqref{eq:64}, 
    \begin{equation}
	\label{eq:67}
	-V_{s,t}=\sum_{j=1}^n \sum_{i=1}^{n-j} F^\dag[\bx^{(j)}_{r,s}\bx^{(i)}_{s,t}]g(y_r)
	+\sum_{j=1}^n R^{r,s}_{n-j}(F^\dag[\bx_{s,t}^{(j)}]g)(y_r).
    \end{equation}
    Combining \eqref{eq:65} with \eqref{eq:66} and \eqref{eq:67}, 
    \begin{equation*}
	R^{r,s}_n(g)+R^{s,t}_n(g)-R^{r,t}_n(g)
	=-\sum_{j=1}^n R^{r,s}_{n-j}(F^\dag[\bx^{(j)}_{s,t}]g). 
    \end{equation*}
    Using Hypothesis~\ref{hyp:10}, 
    \begin{equation*}
	\abs{R^{r,t}_n(g)}\leq \abs{R^{r,s}_n(g)}+\abs{R^{s,t}_n(g)}
	+\sum_{j=1}^n \opnorm{g} k_{n-j}\abs{\bx}\nu_j\omega_{r,s}^{(\gamma+n-j)/p}\omega_{s,t}^{j/p}.
    \end{equation*}

    The Davie lemma (Lemma \ref{lem:davie} in Appendix) allows one to conclude.
\end{proof}
An immediate consequence of Proposition~\ref{prop:decay} combined with the neo-classical
inequality \cite[Theorem 1.2]{hara} is the following control similar to the one
in \cite{boedihardjo} without being restricted to finite dimensional rough paths as no smooth
approximation of the path is used.
In \cite{boutaib}, it is proved that the consistency holds at order $(n+1)/\gamma$ even in presence of an infinite
dimensional rough path. Yet the constant is not identified.
Thanks to the Lyons extension theorem 
(See~\cite[Theorem~2.2.1, p.~242]{lyons98a} or~\cite[Theorem~3.1.2]{lyons02b}), 
\eqref{eq:75} and \eqref{eq:76} are satisfied for rough paths by using a proper choice of the control $\omega$.
Actually, with Example~\ref{ex:product}, Proposition~\ref{prop:decay} may be used to prove the
Lyons extension theorem.
\begin{corollary}[Propagation of factorial decay] 
    Let us fix $n\geq 1$ and $\gamma>0$. Assume that  
    \begin{align}
	\label{eq:75}
	k_j&\leq \frac{B}{\Fact{j/p}}\text{ for }j=0,\dotsc,m\text{ with }m+\gamma>p
	\\
	\label{eq:76}
	\text{ and }
	\nu_j&\leq \frac{KB}{\Fact{j/p}}\text{ for }i=0,\dotsc,n
    \end{align}
    for some $K\geq 0$ with 
    \begin{equation*}
	B\leq \sqrt{p(1-2^{(p-m-\gamma)/p})}.
    \end{equation*}
    Then \eqref{eq:75} is also true for $j=m+1,\dotsc,n$.
\end{corollary}
%%

%%%%%%%%%%%%%%%%%%%%%%%%%%%%%%%%%%%%%%%%%%%%%%%%%%%%%%%%%%%%%%%%%%%%%%
\section{Application to ODEs}

\label{sec:ode}

Let us consider a set of letters $I$.
The algebra $\cW$ is the algebra of words (\textit{i.e.}, monomials) freely generated by $I$. 
The unit element of $\cW$ is the empty word $\emptyset$. The length of a word $w$ is denoted by $\abs{w}$.
By convention, $\abs{\emptyset}=0$.
We define $\cW_{k}:=\Set{w\in\cW\given \abs{w}=k}$, so that $\cW$ is graded by $\Set{\cW_k}_{k\geq 0}$.
We transform any $\RR$-valued family $\Set{m_i}_{i\in I}$ into 
a homomorphism from $\cW$ to $\RR$ by $m_{i_1\dotsb i_k}=m_{i_1}\dotsb m_{i_k}$
for any word $i_1\dotsb i_k$ of $\cW$.
Let us fix $n>0$. Each $i\in I$ is associated to a function $f_i:\uV\to\uV$ of class $\cC^{n}$
and globally Lipschitz. 
We define a Newtonian map $F$ through 
\begin{equation*}
    F[\emptyset]=\ci,\ F[i]=f_i\text{ for }i\in I\text{ and } 
    F[iw](a)=\dD F[w](a)\cdot F[i](a)
\end{equation*}
for any $w\in\cW_{\leq n}$.
For $w\in\cW_{\leq n}$, $F[w]\in\cC^{n+1-\abs{w}}$. In particular, $F[iw]$ is continuous
when $\abs{w}=n$ and $i\in I$.
\begin{proposition}
    We set $\alpha=\sum_{i\in I} m_ii\in\cA$ for $m_i\in\RR$.
    Let $y$ be the unique solution to the ODE 
    \begin{equation*}
	y_t=a+\int_0^t \sum_{i\in I} m_i f_i(y_s)\vd s=a+\int_0^t F[\alpha](y_s)\vd s,\ t\geq 0\text{ for }a\in\uV.
    \end{equation*}
    For $w\in\cW_{\leq n}$, 
    \begin{equation}
	\label{eq:taylor:ode}
	F[w](y_t)=\sum_{k=0}^{n-\abs{w}} \frac{t^k}{k!}\sum_{i_1\dotsb i_k\in I}F[m_{i_1\dotsb i_k} i_1\dotsb i_kw](a)
	+R(t,a)
    \end{equation}
    with 
    \begin{equation*}
	R(t,a)=     \int_0^t \sum_{i_1\dotsb i_{n-\abs{w}}\in I} 
	\frac{(t-s)^{n-\abs{w}}}{(n-\abs{w})!}
	F[m_{i_1\dotsb i_{n+1-\abs{w}}} i_1\dotsb i_{n+1-\abs{w}}w](y_s)\vd s.
    \end{equation*}
\end{proposition}
\begin{proof}
    For $\alpha=\sum_{i\in I} m_i i$, and $t\geq 0$, 
    \begin{equation*}
	\exp(t\alpha)
	=\sum_{k\geq 0}\frac{t^k}{k!} \sum_{i_1\dotsb i_k\in I} m_{i_1}\dotsb m_{i_k} i_1\dotsb i_k
	=\sum_{k\geq 0}\frac{t^k}{k!} \sum_{i_1\dotsb i_k\in I} m_{i_1\dotsb i_k} i_1\dotsb i_k.
    \end{equation*}
    Since $\cW_{\leq n}$ contains only words of size smaller than $n$, 
    \begin{gather*}
	\exp(t\alpha)\prodG w=\sum_{k=0}^{n-\abs{w}}\frac{t^k}{k!}\sum_{i_1\dotsb i_k\in I} m_{i_1\dotsb i_k} i_1\dotsb i_kw
	\\
	\text{and }
	\alpha \prodB \exp((t-s) \alpha)\prodG w=\frac{(t-s)^{n-\abs{w}}}{(n-\abs{w})!}
	\sum_{i_1\dotsb i_{n+1-\abs{w}}\in I} m_{i_1\dotsb i_{n+1-\abs{w}}} i_1\dotsb i_{n+1-\abs{w}}w.
    \end{gather*}
    This concludes the proof.
\end{proof}

We now show that it is also useful to recover some standard results on ODE.

\begin{corollary}[{\cite[Lemma~5.4, p.~85]{hairer_10a}}]
    For $i,j$ such that $f_i,f_j\in\cC^1(\uV,\uV)$, 
    we assume that the vector fields $f_i\dD$ and $f_j\dD$ commutes,
    that is $\dD f_j\cdot f_i-\dD f_i\cdot f_j$ vanishes. Then $\Phi[i]\circ\Phi[j]=\Phi[j]\circ\Phi[i]$.
\end{corollary}
\begin{proof}
    We set $n=2$. For $\alpha,\beta\in\cA$ and $t\geq 0$, 
    a classical computation shows that 
    \begin{equation*}
	\expG(t\alpha)\prodG\expG(t\beta)=\expG\Paren*{t\alpha+t\beta+\frac{t^2}{2}\lie{\alpha,\beta} }
    \end{equation*}
    with $\lie{\alpha,\beta}:=\alpha\beta-\beta\alpha$. 
    We then define $(\alpha\star\beta)(t):=t\alpha+t\beta-\frac{t^2}{2}\lie{\alpha,\beta}$.
    The composition formula yields
    \begin{equation*}
	\Phi[\gamma(t)]=\Phi[t\beta]\circ\Phi[t\alpha]+\epsilon(t)\text{ with }\abs{\epsilon(t)}\leq Ct^3
    \end{equation*}
    as an explicit computation on \eqref{eq:taylor:ode} shows it.
    If $F[ij]=F[ji]$, then $\Phi[(i\star j)(t)]=\Phi[(i+j)t]$. Hence,
    $\abs{\Phi[t\beta]\circ\Phi[t\alpha]-\Phi[t\alpha]\circ\Phi[t\beta]}\leq Ct^3$.
    We conclude using the technique proposed in \cite[Lemma III.5.4, p.~85]{hairer_10a}.
\end{proof}

Similarly, results such as convergence of the Strang splitting and the Lie-Trotter formula \cite{hairer_10a} can be recovered by this approach.  

\begin{remark}
    \label{rem:ode:1}
    A variant consists in assuming that the weights $m_i$ are piecewise continuous functions from $\TT$
    to $\RR$, so that the ODE is a controlled one, where the $m_i$ are the time derivatives
    of the piecewise continuous control $x:\TT\to\RR^{\dim(I)}$. This leads to similar algebraic
    computations.
\end{remark}
\begin{remark}
    When one considers the non-autonomous differential equation $\dot{y}_t=f(t,y_t)$,
    one may fix some integer $n\geq 1$ and set $m_i(t)=\ind{t\in[iT/n,(i+1)T/n]}$
    as well as $f_i(a)=f(iT/n,a)$. Using Remark~\ref{rem:ode:1}, 
    we then obtain discrete time approximations of the chronological expansion
    developed in \cite{agrachev} as well as the Magnus formula \cite{blanes,strichartz,hairer_10a}.
    This also explains why the Magnus formula is sometimes called 
    a \emph{continuous-time Baker-Campbell-Hausdorff-Dynkin formula}. 
\end{remark}
%%

%%%%%%%%%%%%%%%%%%%%%%%%%%%%%%%%%%%%%%%%%%%%%%%%%%%%%%%%%%%%%%%%%%%%%%

\section{High-order expansion of rough differential equations}

\label{sec:rde}

We now consider the situation where $\cW$ is a tensor algebra. In this case,
our notion of RDE is the same as the one by A.~M. Davie \cite{davie05a},
P.~Friz \& N.~Victoir \cite{friz2008,friz}, I.~Bailleul \cite{bailleul12a} as
we saw in \cite{brault1,brault2} that the framework of almost flows encompasses
all these constructions.

Here, we focus on studying high-order expansions of the numerical schemes. 
Such a study was initiated in \cite{friz2008}. It was then followed
by \cite{bailleul12a,bailleul13b,bailleul15a,boedihardjo} and more
recently by~\cite{2003.12626} where Runge-Kutta schemes were studied
or in~\cite{2002.10432} for studying rough transport equation.

\begin{notation}[Tensor algebra]
    We consider a separable Banach space $\uU$ with basis $\Set{e_i}_{i\in I}$.  We
    denote by $\uT(\uU)$ its tensor algebra
    \begin{equation*}
	\cW:=\uT(\uU):=\RR\oplus \uU\oplus \uU^{\otimes 2}\oplus\dotsb. 
    \end{equation*}
    The tensor algebra $\uT(\uU)$ is isomorphic to the algebra of
    words generated by the alphabet $I$. It is also graded by the $\cW_k=\uU^{\otimes k}$,  $k\geq 0$.
\end{notation}

Let us fix $n\geq 1$. 
We associate with each $e_i$ a function $f_i$ of class at most $\cC^n$ and globally Lipschitz.
We consider the differential operators defined by 
\begin{equation*}
    F^\dag[1]g:=g,\ F^\dag[e_i]g:=\dD g\cdot f_i\text{ and }F^\dag[\alpha\otimes \beta]g=F^\dag[\alpha]gF^\dag[\beta]g.
\end{equation*}
Thus, $F^\dag$ defines a Newtonian operator, which is easily put in relation with the Newtonian map
of Section~\ref{sec:ode}.
\begin{notation}[Lie algebra]
    We denote by $\cL$ the smallest Lie 
    algebra generated $\uU$ and close under the Lie brackets $\lie{\alpha,\beta}:=\alpha\otimes\beta-\beta\otimes \alpha$. 
\end{notation}
The proof of the following result is standard. It assert that if $\alpha\in\cL_{\leq n}$, 
then $F[\alpha]$ is actually a first order differential operator. It is proved inductively on the levels 
of brackets in $\ug(\uU)$.  
\begin{lemma}
    \label{lem:hoe:1}
    If $\alpha\in\cL$, then $F^\dag[\alpha]g(a)=\dD g(a)\cdot F^\dag[\alpha]\ci(a)$.
\end{lemma}

We refer to \cite{lyons98a,friz,lyons02b,hairer_kelly} for an extensive discussion
about the properties of weak geometric rough paths.

\begin{definition}[Weak geometric rough path]
    Let $\bx$ be a $\cWln$-valued $p$-rough path (possibly $n=+\infty$ with $\cW_{\leq\infty}=\cW$). 
    We say that $\bx$ is a \emph{weak geometric rough path} if $\log(\bx_{s,t})$ belongs 
    to the Lie algebra.
\end{definition}

The Baker-Campbell-Hausdorff-Dynkin formula \cite{bonfiglioli} states for for any $\alpha,\beta\in\cL$, 
there exists a $\gamma\in\cL$, which is explicitly computable from $\alpha$ and $\beta$ such that 
\begin{equation}
    \label{eq:BCHD}
    \exp(\alpha)\otimes\exp(\beta)=\exp(\gamma)
\end{equation}
for some $\gamma\in\cL$, where $\exp$ is the exponential in the tensor product.
Since $\cW_{> n}$ is a two-sided ideal, by identifying $\cW_{\leq n}$ with 
the quotient space $\cW/\cW_{> n}$, there exists $\gamma_n\in\cL_{\leq n}$ such that 
\begin{equation}
    \label{eq:BCHD:G}
    \expG(\alpha)\prodG\expG(\beta)=\expG(\gamma_n),
\end{equation}
as $\prodG$ is identified with the truncated tensor product $\otimes$. Therefore, Hypothesis~\ref{hyp:303}
holds.

Moreover, when $\alpha\in\cL_{\leq n}$, then $F^\dag[\alpha]$ is a first-order
differential operator by Lemma~\ref{lem:hoe:1} as $\cL$ is stable by projection onto $\cW_{\leq n}$. 
This proves that $F^\dag$ is a Newtonian operator
with respect to the choice of the tensor algebra $\uT(\uU)$ for $\cW$ and its Lie 
algebra for $\cL$.

We then recover the high-order expansion given in \cite{friz2008,boedihardjo}
without relying on the use of sub-Riemannian geodesics. Thus, these
results also hold true in an infinite dimensional setting, as already noted
in \cite{bailleul13b,boutaib}. However, with respect to the two last 
works, we benefits from the decay estimate of \cite{boedihardjo}. 

%\textcolor{red}{Préciser les conditions sur $f$, $\bx$, ...}

\begin{proposition}[Consistency]
    Under the above hypotheses, when $\bx$ is a finite $p$-rough
    paths and $f_i\in\cC^n$ with $n+1>p$, 
    any D-solution $z$ satisfies
    \begin{equation*}
	\abs{z_t-F^\dag[\bx^{(\leq n)}_{s,t}]\ci(z_s)}\leq C\omega_{s,t}^{(n+1)/p}
	\text{ and }
	\abs{z_t-\phi_{t,s}(z_s)}\leq C\omega_{s,t}^{(n+1)/p}
    \end{equation*}
    where $\phi_{t,s}(a)=y[\logG(\bx_{s,t}^{(\leq n)})](a)$.
\end{proposition}

Coupled with the results in \cite{brault2,brault3}, we also obtain a rate 
of convergence of $(n+1)/p-1$ of the numerical scheme. This is the same rate
as in~\cite[Theorem~10.3.3]{friz}.

%%%%%%%%%%%%%%%%%%%%%%%%%%%%%%%%%%%%%%%%%%%%%%%%%%%%%%%%%%%%%%%%%%%%%%
%%%%%%%%%%%%%%%%%%%%%%%%%%%%%%%%%%%%%%%%%%%%%%%%%%%%%%%%%%%%%%%%%%%%%%
%%%%%%%%%%%%%%%%%%%%%%%%%%%%%%%%%%%%%%%%%%%%%%%%%%%%%%%%%%%%%%%%%%%%%%

\section{Branched rough paths}

\label{sec:BRP}

As in Sections~\ref{sec:ode} and \ref{sec:rde}, we consider an alphabet $I$
whose letters $i$ are associated to functions $f_i$. We introduce
first a few notions on trees. 

\begin{definition}[Decorated tree]
    A \emph{$I$-decorated tree} is a tree whose vertices are associated to 
    a letter $i$ in $I$.
    We denote by $\cT$ the set of decorated trees.
\end{definition}
\begin{notation}[Forest]
    The formal, commutative and associative product of trees is called a \emph{forest}. 
    A forest can be seen as a monomial of the $\RR$-module $\RR\bkt{\cT}$ freely generated
    by the set $\cT$. The neutral element of $\RR\bkt{\cT}$, called the \emph{empty tree}, 
    is denoted by $\emptyset$.
\end{notation}
\begin{definition}[Grafting]
    Let us consider $m$ trees $\tau_1\dotsb\tau_m\in\cT$ as well as a decoration $i\in I$.
    We denote by $\Graft{\tau_1,\dotsc,\tau_m}_i$ the tree obtained by grafting the roots of $\tau_1,\dotsc,\tau_m$
    to a new vertex which we decorate with $i$.
    The tree with a single root decorated by $i$ is written~$\bullet_i$.
    The number of nodes of $\tau\in\cT$ is denoted by $\abs{\tau}$, with $\abs{\emptyset}=0$.
    It holds that $\abs{\Graft{\tau_1,\dotsc,\tau_m}}=\abs{\tau_1}+\dotsb+\abs{\tau_m}+1$.
    The gratfting operation $(\tau_1,\dotsc,\tau_m)\mapsto \Graft{\tau_1,\dotsc,\tau_m}$
    is extended to $\RR\bkt{\cT}^m$ by multilinearity.
\end{definition}

Decorated rooted trees are defined recursively through the grafting operation
from the family of single root trees $\Set{\bullet_i}_{i\in I}$.

%%%%%%%%%%%%%%%%%%%%%%%%%%%%%%%%%%%%%%%%%%%%%%%%%%%%%%%%%%%%%%%%%%%%%%

\subsection{B-series}

From the set $\Set{f_i}_{i\in I}$ of functions, we define 
the elementary differentials which is a family of functions indiced
by $I$-decorated rooted trees. The \emph{B} in B-series stands for J.C. Butcher, 
who was the first to exhibit their algebraic features 
to study Runge-Kutta schemes in \cite{Butcher1972}.

\begin{definition}[Elementary differentials]
    To each $I$-decorated rooted tree $\tau$ in $\cT$, we associate the function $F[\tau]:\uV\to\uV$
    by the following way 
    \begin{gather*}
	F[\emptyset](a)=a,\ F[\bullet_i](a)=f_i(a)
	\\
	\text{ and }
	F[\Graft{\tau_1,\dotsc,\tau_m}_i](a)=\dD^m f_i(a)\cdot 
	F[\tau_1](a)\otimes\dotsb\otimes F[\tau_m](a), 
    \end{gather*}
    as long as the $f_i$ have enough regularity to ensure that $F[\tau]$ is continuous.
    We extend $F$ by linearity over $\RR\bkt{\cT}$.
    Such an $F$ is called an \emph{elementary differential}.
\end{definition}

To apply our construction, we define a new algebra $(\cW,+,\cdot)$ from 
$\RR\bkt{\cT}$ with the product defined as a bilinear relation recursively from 
\begin{align*}
    \sigma\cdot\bullet_i&=\Graft{\sigma}_i
    \\
    \text{and }
    \sigma\cdot\Graft{\tau_1,\dotsc,\tau_m}_i
			&=\Graft{\sigma,\tau_1,\dotsc,\tau_m}_i
			+\sum_{j=1}^m \Graft{\tau_1,\dotsc,\sigma\cdot\tau_j,\dotsc,\tau_m}_i.
\end{align*}
This algebra $\cW$ is graded by the number of nodes.

\begin{remark}
    The algebra $(\cW,+,\cdot)$ is non-associative so that $\tau_1\dotsb\tau_m$
    means $\tau_1\cdot (\dotsb(\tau_{m-1}\cdot\tau_m)\cdots)$.
\end{remark}

\begin{hypothesis}
    Let us fix $n\geq 1$ such that for each $\tau\in\cW_{\leq n}$, $F[\tau]\in\cC^1$.
\end{hypothesis}

An immediate consequence of the definition of our product is the following one. 

\begin{lemma}
     The map $F$ is a Newtonian with $\cL=\cW$. 
\end{lemma}
\begin{example}
    \label{ex:bseries}
    Assume that $I$ contains only a single letter, so that we drop any mention to it. Let 
    $y$ be the solution to the ODE $y_t=a+\int_0^t f(y_s)\vd s=a+\int_0^t F[\bullet](y_s)\vd s$. 
    The Taylor formula at order $n$ implies that 
    \begin{equation*}
	y_t=a+F[\expG(t\bullet)](a)+\grandO(t^{n+1}),  
    \end{equation*}
    which in turns means that $y^{(m)}_{|t=0}(a)=F[\bullet^{\cdot m}](a)$, 
    where $y^{(m)}$ is the $m$-th order derivative of $y$ (in time).
\end{example}

For the reason given in Example~\ref{ex:bseries}, $F[\gamma]$ is called
a \emph{B-series}. 

\begin{proposition} Let $\bx:=\Set{\bx_{s,t}}_{(s,t)\in\rTT^2}$ be 
    a $\cW_{\leq n}$-valued $p$-rough path with $n\geq \floor{p}$. 
    Then there exists a D-solution $y$ associated to the almost flow
    $\phi$ defined by $\phi_{t,s}(a):=F[\bx_{s,t}](a)$ and 
    \begin{equation*}
	\abs{y_t-F[\bx_{s,t}](y_s)}\leq C\omega_{s,t}^{(n+1)/p},\
	\text{ for }(s,t)\in\rTT^2.
    \end{equation*}
\end{proposition}

We consider now a Banach space $\uU$ with a basis 
$\Set{e_i}_{i\in I}$ whose elements are then indiced by $I$. 
Now, let $\bx$ be a $\uT_{\leq 2}(\uU)$-valued $p$-rough path with $2\leq p<3$
We consider
a family $\Set{f_i}_{i\in I}$ of functions which are regular enough.

The \emph{Davie approximation} is the almost flow (considered first by A.M.~Davie \cite{davie05a})
defined by 
\begin{equation}
    \label{eq:davie:3}
    \phi_{t,s}(a):=a+\sum_{i\in I}f_i(a)\bx^{(1),i}_{s,t}+\sum_{i,j\in I}\bx^{(2),ij}_{s,t}\dD f_j(a)\cdot f_i(a)
    =F[\bx_{s,t}](a).
\end{equation}
which serves as an approximation to the RDE usually written $z_t=a+\int_0^t f(z_s)\vd\bx_s$, 
which actually means $z_t=a+\int_0^t F[\vd \bx_s](y_s)$. 

The proof of the following theorem is only a rewriting using the formalism of trees.
This is a result similar to the one of T.~Cass and M.~Weidner in \cite{1604.07352} and
the ones of I.~Bailleul in \cite{bailleul20a} where it is shown that the Davie's formalism 
is suitable for branched rough paths. 
With \cite[Proposition~3.8, p.~228]{hairer_kelly}, this leads to the same construction
(only at the first level) of the RDE driven by branched rough paths introduced by \cite{gubinelli10a},
yet with a simpler formulation.

\begin{proposition}
    \label{prop:B1}
    The $\cW_{\leq 2}$-valued family $\tau$ defined by 
    \begin{equation*}
	\btau_{s,t}:=
	1+\sum_{i\in I}\bx^{(1),i}_{s,t}\bullet_i+\sum_{i,j\in I}\bx^{(2),ij}_{s,t}\Graft{\bullet_i}_j.
    \end{equation*}
    is a $\cW_{\leq 2}$-valued $p$-rough path. The Davie approximation defined by~\eqref{eq:davie:3} and rewritten  as
    \begin{equation*}
	\phi_{t,s}(a)=F[\btau_{s,t}](a)
	\text{ for }a\in\uV,\ (s,t)\in\rTT^2
    \end{equation*}
    defines an almost flow. 
\end{proposition}
\begin{proof}
    In $\cW_{\leq 2}$
    (recall that $\prodG$ is nothing more than $\cdot$ in the quotient space $\cW_{\leq 2}$), 
    \begin{equation*}
	\btau_{r,s}\prodG\btau_{s,t}
	=1+\sum_{i\in I}\bullet_i \bx_{r,t}^{(1),i}
	+\sum_{i,j\in I}\Graft{\bullet_j}_i(\bx^{(2),ij}_{r,s}+\bx^{(2),ij}_{s,t})
	+\sum_{i,j\in I}\bullet_i\cdot\bullet_j \bx^{(1),i}_{r,s}\bx^{(2),j}_{s,t},
    \end{equation*}
    since $\bx^{(1)}_{r,s}+\bx^{(1)}_{s,t}=x^{1}_{r,t}$. Regarding the term of order $2$, 
    $\bx^{(2)}_{r,t}=\bx^{(2)}_{r,s}+\bx^{(2)}_{s,t}+\bx^{(1)}_{r,s}\otimes\bx^{(1)}_{s,t}$
    in $\uT_{\leq 2}(\uU)$. We check that $\bullet_i\cdot\bullet_j=\Graft{\bullet_i}_j$. 
    Thus, $\btau$ is multiplicative in $\cW_{\leq 2}$. 

    The controls over $\abs{\btau^{(i)}_{s,t}}$ are easily obtained.
\end{proof}
The main implication of Proposition~\ref{prop:B1} is the convergence of higher-order
numerical schemes constructed from B-series, also called Runge-Kutta schemes.
\begin{corollary}
    Let $n\geq 3$.
    We consider $\btau$ given in Proposition~\ref{prop:B1}.
    Any $\cW_{\leq n}$-valued $p$-rough path $\bsigma$ with 
    $\bsigma^{(\leq 2)}_{s,t}=\btau_{s,t}$ for any $(s,t)\in\rTT^2$
    gives rise to the same flow as the one that stems from the approximation $\Set{F[\btau_{s,t}]}_{(s,t)\in\rTT^2}$
\end{corollary}

Finally, in $\cW_{\leq 2}$, 
\begin{equation*}
    \log(\btau_{s,t})
    =
    \sum_{i\in I}\bullet_i \bx_{r,t}^{(1),i}
    +\sum_{i,j\in I}\Graft{\bullet_j}_i(\bx^{(2),ij}_{r,s}+\bx^{(2),ij}_{s,t})
    -\frac{1}{2}\sum_{i,j\in I}\bx^{(1),i}\bx^{(1),j}\Graft{\bullet_i}_j.
\end{equation*}

The next lemma is a rewriting in the context of branched trees on the 
definition of weak geometric rough paths \cite{lyons02b,friz,lejay_victoir}.

\begin{lemma}
    \label{lem:B1}
    A weak geometric $p$-rough path is a rough path $\bx$ such 
    that $\bx_{s,t}=\exp(\blambda_{s,t})$ where for any $(s,t)\in\rTT^2$, 
    $\blambda_{s,t}$
    belongs to the Lie algebra generated by the Lie brackets in $\uT_{\leq n}(\uU)$.
    In $\cW_{\leq 2}$, 
    \begin{equation}
	\label{eq:B1}
	\bx_{s,t}=1+\bx^{(1)}_{s,t}+\frac{1}{2}\bx^{(1)}_{s,t}\otimes\bx^{(1)}_{s,t}
	+\ba_{s,t}
    \end{equation}
    where $\ba_{s,t}\in\uU\otimes\uU$ is anti-symmetric, that is $\ba^{ij}_{s,t}=-\ba^{ji}_{s,t}$
    for any $i,j\in I$.
\end{lemma}

It follows from Lemma~\ref{lem:B1} that if $\btau_{s,t}$ proceeds from a weak geometric $\uT_{\leq 2}(\uU)$-valued
rough path as in \eqref{eq:B1}, then 
\begin{equation*}
    \blambda_{s,t}:=\logG(\btau_{s,t})=\bx^{(1)}_{s,t}+\sum_{i,j\in I}\Graft{\bullet_j}_i\ba^{ij}_{s,t}.
\end{equation*}
One recovers the construction from I.~Bailleul \cite{bailleul12a,bailleul13b} when one considers 
as an almost flow $\Set{\Phi[\blambda_{s,t}]}_{(s,t)\in\rTT^2}$.

Let us end this section by noting that 
for any $\cW_{\leq n}$-valued $p$-rough path $\btau$, 
it is always possible to compute $\blambda_{s,t}=\logG(\btau_{s,t})$, 
whether or not the path is a  weak geometric rough path. This has to 
be connected with the results of M.~Hairer and D.~Kelly \cite{hairer_kelly}.

%%%%%%%%%%%%%%%%%%%%%%%%%%%%%%%%%%%%%%%%%%%%%%%%%%%%%%%%%%%%%%%%%%%%%%%%%%%%%%%

%%
\subsection{Aromatic rough paths}

The notion of aromatic tree extends the one of rooted trees.
Aromatic trees were introduced in \cite{MT2016,MR3510021}
as an extension of B-series to study equivariant numerical schemes.

\begin{definition}[Directed graph]
    A \emph{directed graph} $g$ is defined as a finite set of vertices $V(g)$ and edges $E(g)\subset V(g)\times V(g)$. 
    We allow the empty graph $\emptyset$ with no vertices.
    Two graphs $g$ and $h$ are \emph{equivalent} is there exists a one-to-one map $f:V(g)\to V(h)$
    such that $(f\times f)(E(g))=E(h)$.
\end{definition}
\begin{definition}[Aromatic forest]
    An \emph{aromatic forest} is an equivalence class of directed graphs where 
    each node has at most one outgoing edge. A \emph{root} of an aromatic forest
    is a node with no outgoing edge. 
\end{definition}

A directed graph may have several connected components.
An aromatic forest is then an unordered collection of connected components,
each one with exactly one root (\emph{rooted tree}) or no root so that it contains
one cycle called an \emph{aroma}. This is how it differs from the construction
of Section~\ref{sec:BRP}. A grading is naturally induced by the number 
of vertices.

We now consider that $\uV:=\RR^m$. Let $I$ be a set of letters. 
To each letter $\alpha\in I$, we associate 
a function $f[\alpha]:\uV\to\uV$ regular enough.
The $i$-th component of $f[\alpha]$ is written $f^i[\alpha]$, $i=1,\dotsc,m$.

\begin{definition}[Elementary differential operator]
    We consider a $I$-decorated aromatic forest $\tau$. 
    We associate with $\tau$ a differential operator $F^\dag[\tau]$ constructed the following way:
    \begin{itemize}[noitemsep,topsep=-\parskip,partopsep=0pt]
	\item $F^\dag[\emptyset]$ is the identity operator. 
	\item Each vertex $v$ of $\tau$ is also associated with some
	    dummy indice $i\in\Set{1,\dotsc,m}$
	    and thus decorated by $(\alpha,i)\in I\times\Set{1,\dotsc,m}$.
	\item Let $v$ be a vertex of $\tau$ with $s$ ingoing edges 
	    from the vertices labelled by $(\alpha_1,i_1),\dotsc,(\alpha_s,i_s)$.
	    We associate the factor
	    \begin{equation*}
		f(\alpha,i,i_1,\dotsc,i_s):=\frac{\partial^s f^{i}[\alpha]}{\partial x_{i_1}\dotsb \partial x_{i_s}}.
	    \end{equation*}
	    If the vertex has no ingoing edge, then $f(\alpha,i):=f^i[\alpha]$. 
	\item To each root of $\tau$ labelled by $(\alpha,i)$, 
	    we associate the factor $f(\alpha,i):=\frac{\partial}{\partial x_i}$.
	\item To construct $F^\dag[\tau]$,
	    we multiply the factors and we sum according the Einstein summation convention, 
	    \textit{i.e.}, we sum over each dummy indice appearing twice, 
	    lower (differentiation) and upper (choice of component). 
    \end{itemize}
\end{definition}

The order of the differential operator $F^\dag[\tau]$ depends on the number of roots
in general.

\begin{example}
    The tree $\bullet_{\alpha}$ which gives rise to the vector field $F^\dag[\bullet_\alpha]:=\sum_{i=1}^m f^i[\alpha] \frac{\partial}{\partial x_i}$,
    while
    \begin{equation*}
	\Graft{\bullet_\alpha}_{\beta} \text{ gives rise to }
	F^\dag[\Graft{\bullet_\alpha}_{\beta}]:=
	\sum_{i,j=1}^m f^i[\alpha]\frac{\partial f^j[\beta]}{\partial x_i}\frac{\partial}{\partial x_j}
    \end{equation*}
    and
    \begin{equation*}
	\Graft{\bullet_\alpha \bullet_\gamma}_{\beta} \text{ gives rise to }
	F^\dag[\Graft{\bullet_\alpha \bullet_\gamma}_{\beta}]:=
	\sum_{i,j,k=1}^m f^i[\alpha]f^j[\gamma]\frac{\partial^2 f^k[\beta]}{\partial x_i\partial{x_j}}\frac{\partial}{\partial x_k}.
    \end{equation*}
\end{example}

\begin{example}
    The aromatic forest $\bullet_\alpha\bullet_\beta$
    gives rise to the second-order differential operator
    $F^\dag[\bullet_\alpha\bullet_\beta]=\sum_{i,j} f^i[\alpha]f^j[\beta]\frac{\partial^2}{\partial x_i\partial{x_j}}$.
\end{example}

\begin{example}
    An \emph{aroma} is a loop with a single edge both ingoing and outgoing.
An aroma with a single vertex is associated with the divergence $\sum_i \frac{\partial f_i[\alpha]}{\partial x_i}$. It is a differential operator of order $0$.
\end{example}

\begin{remark}
    \label{rem:aromatic:1}
    Elementary differentials which are differential operators of first order 
    are encoded by aromatic forest with a single root.
\end{remark}

\begin{definition}[Composition of $I$-decorated aromatic forest]
    Given two aromatic $I$-decorated aromatic forests, we define their products $\sigma\cdot\tau$ 
    as the sum of the graphs~$\gamma$ obtained the following way:
    \begin{itemize}[noitemsep,topsep=-\parskip,partopsep=0pt]
	\item We superpose $\sigma$ and $\tau$, that is we construct a graph whose edges (resp. vertices)
	    is obtained by the union of the edges (resp. vertices) of $\sigma$ and $\tau$.
	\item We form new graphs by considering all the possible 
	    way to add edges from roots of $\sigma$ to nodes of $\tau$. 
    \end{itemize}
\end{definition}

As the product of $\sigma$ and $\tau$ consists in adding all graphs obtained by 
considering all possible edges outgoing from the roots of $\sigma$, 
one may check that the product is associative (but not commutative). 
Besides, the Lie brackets between two elements each with one root 
as one root.

\begin{example}
    The product of $\sigma=\Graft{\bullet_\alpha}_\beta$ with 
    $\tau=\Graft{\bullet_\gamma}_{\delta}$ is 
    \begin{equation*}
\sigma\cdot \tau
=
	\Graft{\bullet_\alpha}_\beta\Graft{\bullet_\gamma}_{\delta}
	+
	\Graft{\Graft{\Graft{\bullet_\alpha}_\beta}_\gamma}_{\delta}
	+
	\Graft{\Graft{\bullet_\alpha}_\beta,\bullet_\gamma}_{\delta}.
    \end{equation*}
    The corresponding differential operator is
    \begin{multline*}
	F^\dag[\sigma\cdot\tau]
	=
	\sum_{i,j,k,\ell=1}^m f^i[\alpha]\frac{\partial f^j[\beta]}{\partial x_i}
	f^k[\gamma]\frac{\partial f^\ell[\delta]}{\partial x_k}
	\frac{\partial^2}{\partial x_j\partial x_\ell}
	\\
	+\sum_{i,j,k,\ell=1}^m
	f^i[\alpha]\frac{\partial f^j[\beta]}{\partial x_i}
	\frac{\partial f^k[\gamma]}{\partial x_j}\frac{\partial f^\ell[\delta]}{\partial x_k}
	\frac{\partial}{\partial x_\ell}
	+
	\sum_{i,j,k,\ell=1}^m
	f^i[\alpha]\frac{\partial f^j[\beta]}{\partial x_i}
	f^k[\gamma]\frac{\partial^2 f^\ell[\delta]}{\partial x_j\partial x_k}
	\frac{\partial}{\partial x_\ell}.
    \end{multline*}
\end{example}

The product rule (see \cite[Definition~3.7, p.~205]{bogfjellmo}) is
the one that corresponds to the product of two elementary differentials. 
The next lemma follows from \cite[Lemma~3.8, p.~205]{bogfjellmo}
and Remark~\ref{rem:aromatic:1} above.

\begin{lemma} 
   Let $\cL$ be the $\RR$-submodule freely generated by $I$-decorated rooted trees with one root.
    Assume that $F^\dag[\alpha]\ci$ is Lipschitz continuous for any $\alpha\in\cL$
    and $F^\dag[\beta]$ of coefficients of class $\cC^1$ for any $\beta\in\cW$
    where $\cW$ is the set of all $I$-decorated trees. 
    Then~$F^\dag$ is a Newtonian operator.
\end{lemma}

It is then possible to extend the notion of RDE driven by branched rough
paths to RDE driven by aromatic rough paths.  

%%%%%%%%%%%%%%%%%%%%%%%%%%%%%%%%%%%%%%%%%%%%%%%%%%%%%%%%%%%%%%%%%%%%%%
%--------------------------------------------------%
\appendix
%--------------------------------------------------%

\section{A Davie lemma}

\label{sec:davie}

There are several forms of the Davie lemma. The one presented here is a simplified form 
of the one in \cite{brault2} and of \cite[Lemma~10.59, p.~270]{friz} for the special case $\varpi(x)=x^\theta$, $\theta>1$.

\begin{lemma}[A Davie's lemma]
    \label{lem:davie}
    Let $\varpi$ be such that $\varpi(0)=0$ and $2\varpi(x)\leq \kappa\varpi(2x)$ for some $\kappa<1$.
    Let $\Set{U_{s,t}}_{0\leq s\leq t\leq T}$ be a $\RR_+$-valued family 
    as well as $E,M\geq 0$ such that 
    \begin{equation*}
	U_{s,t}\leq E\varpi(\omega_{s,t})\text{ and }
	U_{r,t}\leq U_{r,s}+U_{s,t}+M\varpi(\omega_{r,t})
    \end{equation*}
    for any $0\leq r\leq s\leq t\leq T$.  Then 
    $U_{r,t}\leq M/(1-\varkappa)\varpi(\omega_{r,t})$. 
\end{lemma}

%%%%%%%%%%%%%%%%%%%%%%%%%%%%%%%%%%%%%%%%%%%%%%%%%%%%%%%%%%%%%%%%%%%%%%
%%%%%%%%%%%%%%%%%%%%%%%%%%%%%%%%%%%%%%%%%%%%%%%%%%%%%%%%%%%%%%%%%%%%%%
%%%%%%%%%%%%%%%%%%%%%%%%%%%%%%%%%%%%%%%%%%%%%%%%%%%%%%%%%%%%%%%%%%%%%%

\paragraph*{Acknowledgement.} The author thanks
Laure Coutin and Antoine Brault for interesting discussions related
to this article, and the referees for their valuable comments.

%%%%%%%%%%%%%%%%%%%%%%%%%%%%%%%%%%%%%%%%%%%%%%%%%%%%%%%%%%%%%%%%%%%%%%
%%%%%%%%%%%%%%%%%%%%%%%%%%%%%%%%%%%%%%%%%%%%%%%%%%%%%%%%%%%%%%%%%%%%%%
%%%%%%%%%%%%%%%%%%%%%%%%%%%%%%%%%%%%%%%%%%%%%%%%%%%%%%%%%%%%%%%%%%%%%%
%\printbibliography

\end{document}